\documentclass{article}
\usepackage{amsmath, amsthm, amssymb, setspace, graphics}
\usepackage[all]{xy}
\usepackage{url}
\newtheorem{thm}{Theorem}[subsection] 
\newtheorem{pro}[thm]{Proposition} 
\newtheorem{lem}[thm]{Lemma} 
\newtheorem{cor}[thm]{Corollary} 
\newtheorem{conj}[thm]{Conjecture}
\theoremstyle{definition} 
\newtheorem{defn}[thm]{Definition} 
\theoremstyle{remark} 
\newtheorem{rem}[thm]{Remark}
\newtheorem{conv}[thm]{Convention}
\newtheorem{notn}[thm]{Notation}
\newtheorem{exa}[thm]{Example}

\newcommand{\CC}{\mathbb C}

\newcommand{\NN}{\mathbb N}
\newcommand{\PP}{\mathbb P}
\newcommand{\QQ}{\mathbb Q}

\newcommand{\UU}{\mathbb U}

\newcommand{\ZZ}{\mathbb Z}

\newcommand{\Ocal}{\mathcal O}
\newcommand{\Ccal}{\mathcal C}
\newcommand{\Lcal}{\mathcal L}
\newcommand{\Hcal}{\mathcal H}
\newcommand{\Fcal}{\mathcal F}
\newcommand{\Tcal}{\mathcal T}
\newcommand{\Vcal}{\mathcal V}

\newcommand{\Dcal}{{\mathcal D}}
\newcommand{\Scal}{{\mathcal S}}
\newcommand{\iii}{{h}}
\newcommand{\rest}[1]{{\textstyle|}_{#1}}

\newcommand{\pt}{\operatorname{pt}}

\newcommand{\mult}{\operatorname{mult}}

\newcommand{\age}{\operatorname{age}}
\newcommand{\Res}{\operatorname{Res}}
\newcommand{\al}{\alpha}

\newcommand{\fie}{\varphi}

\newcommand{\ga}{\gamma}
\newcommand{\la}{\lambda}

\newcommand{\ev}{{\rm ev}}

\newcommand{\ctop}{{c}_{\mathrm{top}}}

\newcommand{\cxi}{\mathtt{i}}

\newcommand{\irr}{\operatorname{irr}}

\newcommand{\tw}{{\rm tw}}
\newcommand{\un}{{\rm un}}
\newcommand{\ch}{{\rm ch}}

\newcommand{\textsum}{{\textstyle{\sum}}}

\providecommand{\abs}[1]{\lvert#1\rvert}

\def\vir{{\rm vir}}

\def\pmmu{{\pmb \mu}}

\def\W{\text{W}}
\def\RW{\operatorname{FJRW}}
\def\GW{\operatorname{GW}}

\def\ol{\overline}

\def\wt{\widetilde}

\newcommand{\MMM}{\overline{\mathcal M}}

\newcommand{\RRR}{\mathcal R}
\begin{document}

\title{
\textbf{Landau--Ginzburg/Calabi--Yau correspondence for quintic
three-folds via symplectic transformations}}
\author{Alessandro Chiodo and Yongbin Ruan}
\date{ }
\maketitle

\begin{abstract}
{We compute the recently introduced Fan--Jarvis--Ruan--Witten
theory of $W$-curves in genus zero for quintic polynomials in five
variables and we show that it matches the Gromov--Witten
genus-zero theory of the quintic three-fold via a symplectic
transformation. More specifically, we show that the $J$-function
encoding the Fan--Jarvis--Ruan--Witten theory on the A-side equals
via a mirror map the $I$-function embodying the period integrals
at the Gepner point on the B-side. This identification inscribes
the physical Landau--Ginzburg/Calabi--Yau correspondence within
the enumerative geometry of moduli of curves, matches the
genus-zero invariants computed by the physicists Huang, Klemm, and
Quackenbush at the Gepner point, and yields via Givental's
quantization a prediction on the relation between the full higher
genus potential of the quintic three-fold and that of
Fan--Jarvis--Ruan--Witten theory.}
\end{abstract}

\vspace*{6pt}\tableofcontents

\section{Introduction}\label{sect:intro}
    During the last twenty years, mirror symmetry has
    been one of the most inspirational problems arising from physics.
    There are various formulations of mirror symmetry and each one
    is important in its
    own way. The most classical version proposes
    a conjectural duality in the context of
    Calabi--Yau (CY) complete intersections of toric varieties,
    which interchanges quantum cohomology with
    the Yukawa coupling of the variation of the Hodge structure. In
    particular, it yields a striking prediction of the genus-zero
    Gromov--Witten  invariants encoding the enumerative geometry
    of stable maps from curves to these CY manifolds. The
    most famous example is the quintic three-fold defined by a
    single degree-five homogeneous polynomial
    $$W=x^5_1+x^5_2+x^5_3+x^5_4+x^5_5,$$
    for which the genus zero predictions have been completely proven
    \cite{Gi} \cite{LLY}.

    In the early days of mirror symmetry, physicists
    noticed that the defining equations of CY
    hypersurfaces or complete intersections---such as the above quintic
    polynomial---appear naturally in another context; namely, the
    Landau--Ginzburg (LG) singularity model. The argument has been made on physical
    grounds \cite{VW} \cite{Wi3} that there should be a
    Landau--Ginzburg/Calabi--Yau (LG/CY) correspondence connecting
    CY geometry to the
    LG singularity model. In this context,
    CY manifolds are considered from the point of view of
    Gromov--Witten theory; this correspondence would therefore inevitably
    yield new predictions on Gromov--Witten invariants and is likely
    to greatly simplify their calculation (it is generally believed that
    the LG singularity model is easier to compute). Here, we
    should mention that in
    Gromov--Witten theory several general methods
    have been recently found \cite{LR} \cite{MP}  and
    can in principle determine Gromov--Witten invariants in a wide range of cases.
    These methods, however, are hard to put into practice both
    when calculating a single invariant and when
    one needs to effectively compute the full higher genus
    Gromov--Witten theory.
    The genus-one theory has been
    computed only recently by A.~Zinger after a great deal of hard
    work \cite{Zi}. Computations in higher genera are out of
    mathematicians' reach for the
    moment. There is a physical method put forward
    by Huang--Klemm--Quackenbush \cite{HKQ}. They worked in the B-model, \emph{i.e.} on
    the complex moduli space of the mirror quintic three-fold,
    and provided the A-model predictions from mirror symmetry: their computations
    interestingly combines the potential at the ``large complex structure  point'' (mirror
    of the Calabi--Yau quintic) with the ``Gepner point''
    potential.
    The regularity of the latter
    yields predictions for the quintic three-fold up to $g=51$. Unlike the ``large complex structure point''
    the invariants
    encoded by this ``Gepner point'' potential lack a
    mirror geometrical interpretation in terms of
    enumerative geometry
    of curves.

    In this current unsatisfactory state of affairs, a natural
    idea is to push through the LG/CY
    correspondence in enumerative geometry of curves and
    use the computational power of the LG singularity model as an effective
    method for determining the higher genus Gromov--Witten invariants of the quintic
    three-fold. At
    first sight, it seems surprising that this idea
    has escaped attention for twenty years;
    however, a
    brief investigation reveals that the problem is far more
    subtle than one might expect. To begin with, the
    LG/CY correspondence is
    a physical statement concerning conformal field
    theory and lower energy effective theories: it does not
    directly imply an explicit geometric prediction.
    At a more fundamental level, Gromov--Witten
    theory embodies all the relevant information on the CY-side,
    whereas
    it is not clear which theory plays the
    same role on the LG-side.
    Identifying such a counterpart to Gromov--Witten theory is the
    first step towards establishing a geometric LG/CY correspondence and
    is likely to be interesting in its own right.
    For instance, in a different context, the LG/CY correspondence led to identify
    matrix factorization as the LG counterpart of the derived category
    of complexes of coherent
    sheaves of CY varieties \cite{H} \cite{Orl} \cite{Ko1}.

    In \cite{FJR1, FJR2, FJR3}, a candidate quantum theory of singularities
    has been constructed.
    The formulation of this theory is very
    different from Gromov--Witten theory.
    Naively, Gromov--Witten theory can be thought of
    as solving the Cauchy--Riemann
    equation $\ol{\partial}f=0$ for
    the map $f\colon \Sigma\rightarrow
    X_W$, where $\Sigma$ is a compact Riemann surface and
    $X_W$ is the weighted projective hypersurface $\{W=0\}$.
    In many ways, the difficulty and the interest of
    the computation of Gromov--Witten invariants comes from the
    fact that $X_W$ is a
    nonlinear space. On the other hand,
    the polynomial $W$ in $N$ variables in the LG singularity model
    is treated as
    a holomorphic function on $\CC^N$. Since there
    are no nontrivial holomorphic maps from a compact Riemann
    surface to $\CC^N$, we run into difficulty at a
    much more fundamental level. The solution comes from
    regarding singularities from a rather different point of view.
    In the early 90's,
    Witten conjectured \cite{Wi1} and Kontsevich proved \cite{Ko} that the intersection theory
    of Deligne and Mumford's moduli of curves
    is governed by the KdV integrable
    hierarchy---\emph{i.e.} the integrable system corresponding to the $A_1$-singularity.
    Witten also pursued the generalization of
    Deligne--Mumford spaces to new moduli spaces governed by
    integrable hierarchies attached to other singularities.
    To this end, he proposed, among other approaches,
    a remarkable partial differential equation of the form
    $$\ol{\partial}s_j+\overline{{\partial_jW}(s_1, \cdots, s_N)}=0,$$
    where $W$ is a quasihomogeneous polynomial and $\partial_jW$ stands for the
    partial derivative with respect to the $j$th variable. A
    comprehensive moduli theory of the above Witten equation has
    been established by Fan, Jarvis, and Ruan \cite{FJR1, FJR2, FJR3}.
    Besides extending Witten's 90's conjecture
    (see \cite{FSZ} for $A_n$-singularities and \cite{FJR1} for all
    simple singularities),
    one outcome of Fan--Jarvis--Ruan--Witten  theory
    is that it plays the role of Gromov--Witten theory
    on the LG-side for any quasihomogeneous
    singularity. In this perspective, the Witten equation
    should be viewed as the counterpart in the LG-model to the
    Cauchy--Riemann equation: we replace a linear equation on a
    nonlinear target with a nonlinear equation on a linear target.

    With the LG-counterpart of Gromov--Witten theory understood, there are
    two remaining issues: $(i)$ Is this LG-side really easier to
    compute? $(ii)$ What is the precise mathematical statement for the
    LG/CY correspondence in terms of Gromov--Witten theory?
    For $(i)$, there is ample evidence that this is
    indeed the case. For example, Fan, Jarvis, and Ruan computed their
    theory for ADE-singularities, thereby  establishing
    Witten's original conjecture claiming that the theory is governed by
    the ADE-hierarchies.
    For $(ii)$, motivated by a similar conjecture for crepant resolution of orbifolds
    \cite{CR} and \cite{CIT}, a natural conjecture \cite{R} can be formulated in terms of
    Givental's Lagrangian cones and quantization. In this paper we
    state it for the quintic three-fold
    and we establish it in genus zero via a symplectic transformation.
    We achieve this, by computing Fan--Jarvis--Ruan--Witten theory
    for the quintic polynomial singularity $\{W=0\}$; the result
    supplies a geometrical interpretation in terms of
    enumerative geometry of curves
    of Huang, Klemm, and Quackenbush's genus-zero potential at the
    ``Gepner point'' (Remark \ref{rem:HKQ}).
    Furthermore, the formula for the symplectic transformation $\mathbb U$,
    makes the higher genus statement explicit,
    for $g\ge 0$ the correspondence is expected to
    be carried out via Givental's quantization of $\mathbb U$.

\subsection{The main result}\label{sect:mainres}
    The LG/CY correspondence concerns quasihomogeneous, or weighted homogeneous,
    polynomials $W$ in $N$ variables:
    \begin{equation}\label{eq:quasihomindices}
    \exists c_1,\dots,c_N,d>0\ \mid \qquad W(\lambda^{c_1}x_1, \dots, \lambda^{c_N}x_N)=\lambda^{d} W(x_1,
    \dots, x_N) \qquad \forall \la \in \CC.\end{equation}
    The geometric object considered on the CY side is
    a Calabi--Yau hypersurface inside the weighted projective space $\PP(c_1,\dots,c_N)$.
    On the LG-side,
    we regard $W$ as the equation of
    an isolated singularity in the
    affine space $\CC^{N}$.
    In this paper, we focus on the Fermat quintic polynomial $x_1^5+x_2^5+x_3^5+x_4^5+x_5^5$.

    The Fan--Jarvis--Ruan--Witten
    genus-zero invariants of the $W$-singularity (LG side)
    $$\langle \tau_{a_1}(\phi_{\iii_1}),\dots,\tau_{a_{n-1}}(\phi_{\iii_{n-1}}),
    \tau_{a_n}(\phi_{\iii_{n}})\rangle^{\RW}_{0,n}$$
    and the Gromov--Witten genus-zero invariants of the $W$-hypersurface  (CY side)
    $$\langle \tau_{a_1}(\fie_{\iii_1}),\dots,\tau_{a_{n-1}}(\fie_{\iii_{n-1}}),
    \tau_{a_n}(\fie_{\iii_{n}})\rangle^{\GW}_{0,n,\delta}$$
    are intersection numbers defined (see \eqref{eq:FJRWinv} and \eqref{eq:GWinv})
    for any $(a_1,\dots, a_n)\in \NN^n$ and for any
    entry of the
    state spaces of the two theories: $H_{\RW}=\oplus_\iii \phi_\iii \CC$ and $H_{\GW}=\oplus_\iii \fie_\iii \CC$.
    These two sets of numbers
    arise from two entirely different problems
    of enumerative geometry of curves.
    The definition of the
    Gromov--Witten invariants $\langle \quad \rangle ^{\GW}$
    is well known: the essential ingredient is
    the moduli space of stable maps to the Calabi--Yau variety.
    The Fan--Jarvis--Ruan--Witten invariants
    $\langle \quad \rangle ^{\RW}$ have been recently introduced \cite{FJR1}
    and are based on a generalization of Witten's
    moduli space parametrizing $d$-spin curves: curves equipped with
    a line bundle $L$, which is a $d$th root of the cotangent
    bundle\footnote{For marked curves we should rather refer to
    the log-canonical
    bundle obtained by twisting the dualizing sheaf at the markings.} (\emph{i.e.} $L^{\otimes d}=\omega$).
    The general moduli problem considered in \cite{FJR1} is as follows:
    for $W$ satisfying \eqref{eq:quasihomindices},
    we endow the curve with $N$ line bundles $L_1,
    \dots, L_N$ which are $d$th roots of
    $\omega^{\otimes c_1},\dots, \omega^{\otimes c_N}$
    (the line bundles satisfy further
    relations in terms of $W$; namely for each monomial $x_1^{m_1}\dots x_N^{m_N}$ of
    $W$ we require $L_1^{\otimes m_1}\otimes \dots\otimes L_N^{\otimes m_N}\cong \omega$).
    We specify this definition in Section \ref{sect:FJRW} in the case of the Fermat quintic polynomial.

    The two sets of invariants can be incorporated into
    the Fan--Jarvis--Ruan--Witten partition function and into the
    Gromov--Witten partition function, which, by standard
    techniques,
    can be
    reconstructed from the
    generating functions of the
    one-point descendants:
    the invariants with not more than one entry $\tau_a(\phi_\iii)$
    and $\tau_a(\fie_\iii)$
    having $a\ne0$:
    $\langle \tau_{0}(\phi_{\iii_1}),\dots,\tau_{0}(\phi_{\iii_{n-1}}),
    \tau_a(\phi_{\iii_{n}})\rangle^{\RW}$
    and
    $\langle \tau_{0}(\fie_{\iii_1}),\dots,\tau_{0}(\fie_{\iii_{n-1}}),
    \tau_a(\fie_{\iii_{n}})\rangle^{\GW}.$
    In other words, the two theories are determined by the
    $J$-functions
    \begin{align}\label{eq:JforRW}
    J_{\RW}(\textsum_{\iii} t^\iii_0 \phi_\iii,z)&=z\phi_0+\sum _{\iii}t_0^\iii \phi_\iii +
    \sum_{\substack{n\ge 0 \\ {(\iii_1,\dots, \iii_n)}}}
    \sum_{\substack{ \epsilon, k}}\frac{t_{0}^{\iii_1}\cdots t_{0}^{\iii_n}}{n!z^{k+1}}
    \langle \tau_{0}(\phi_{\iii_1}),\dots,\tau_{0}(\phi_{\iii_n}),\tau_k(\phi_\epsilon)
    \rangle^{\RW}_{0,n+1}\phi^\epsilon,\\
    J_{\GW}(\textsum_{\iii} t^\iii_0 \fie_\iii,z)&=z\fie_0+
    \sum _{\iii}t_0^\iii \fie_\iii +
    \sum_{\substack{n\ge 0\\ \delta \ge 0\\
    {(\iii_1,\dots, \iii_n)}}}
    \!\!\!\!\sum_{\substack{ \epsilon, k }}\frac{t_{0}^{\iii_1}\cdots t_{0}^{\iii_n}}{n!z^{k+1}}
    \langle \tau_{0}(\fie_{\iii_1}),\dots,\tau_{0}(\fie_{\iii_n}),\tau_k(\fie_\epsilon)
    \rangle^{\GW}_{0,n+1,\delta}\fie^\epsilon,
    \end{align}
    which can be regarded as terms of
    $H_{\RW}((z^{-1}))$ and $H_{\GW}((z^{-1}))$,
    \emph{i.e.}
    Laurent series with coefficients in $H_{\RW}$ and $H_{\GW}$.

    The LG/CY correspondence claims that the two $J$-functions can be
    determined from each other.
    The correspondence relies on an
    isomorphism at the level of state spaces $H_{\RW}$ and $H_{\GW}$ on which the
    two $J$-functions are defined.
    In fact, the state space
    $H_{\GW}$,
    the Chen--Ruan cohomology of the hypersurface, and the state space $H_{\RW}$,
    an orbifold-type Milnor ring $\mathcal H_{\RW}(W)$ (see equation \eqref{eq:sspace}),
    are isomorphic for all
    weighted CY hypersurfaces \cite{ChiR}.
    In this paper we only need the quintic three-fold case, where
    the isomorphism can be easily shown, see Example \ref{exa:isosspace}.
    In fact, for $W=x_1^5+x_2^5+x_ 3^5+x_4^5+x_5^5$,
    further simplifications occur: for dimension reasons (see Sections \ref{sect:enum} and
    \ref{sect:setup}),
    the two
    $J$-functions can be uniquely determined by their restrictions to
    two lines lying in the degree-2 part of the respective
    state spaces identified via the above isomorphism: the line
    $t_0^1\phi_1$ and the line $t_0^1\fie_1$. In this way, the restriction takes the form of a one parameter
    function. In Fan--Jarvis--Ruan--Witten theory, we have $$J_{\RW}(t_0^1\phi_1,z)=z\phi_0+t_0^1\phi_1+ \sum_{n\ge 0}
     \sum_{\substack{ \epsilon, k}}\frac{(t_{0}^{1})^n}{n!z^{k+1}}
    \langle \tau_{0}(\phi_{1}),\dots,\tau_{0}(\phi_1),\tau_k(\phi_\epsilon)
    \rangle^{\RW}_{0,n+1}\phi^\epsilon.$$ The analogue one-parameter expression for the $J$-function
    also holds on the line $t_0^1\fie_1$ in
    Gromov-Witten theory; we refer to the discussion in Section 3.1, in particular
    Remarks \ref{rem:divisorequation} and
    \ref{rem:CYsimpl}, where we illustrate how in this case the parameter $t_0^1$ only appears in the form
    $q=\exp(t_0^1)$.

    On the CY side,
    Givental's mirror symmetry theorem \cite{Gi}  for the
    quintic three-fold
    sets an equivalence between the above $J$-function and
    the $H_{\GW}((z^{-1}))$-valued
    $I$-function
    $$
    I_{\GW}(q,z)=
     \sum_{d\ge 0}zq^{H/z+d}
    \frac{\prod_{k=1}^{5d}(5H+kz) }{\prod_{k=1}^d (H+kz)^5},$$
where $H$ is the cohomology class corresponding to the hyperplane
    section, $q^{H/z}$ should be
read as the expansion $\exp(H\log(q)/z)$ in the cohomology ring,
     and  $q=\exp({t_0^1})$ parametrizes the line
    $\CC\fie_1$ as already mentioned (it is usually regarded as a parameter centred at the
    ``large complex structure point'' of the complex moduli space of
    the quintic three-fold).
    Expanded in the variable $H$, the
    $I$-function assembles the period integrals spanning the space of
    solutions of Picard--Fuchs equation
    \begin{align*}
    &\left[{D_q}^4-5q\prod_{m=1}^4(5D_q+mz)\right]I_{\GW}=0 &&\text{\Bigg(for $D_q=zq\frac{\partial }{\partial q}$\Bigg).}
    \end{align*}
(These period integrals were first systematically studied
in this context by Candelas, De La Ossa, Green, and Parkes in \cite{CDGP}.)
Via an explicit change of variables
    \begin{align*}
    &\log q'=\frac{g_{\GW}(q)}{f_{\GW}(q)}
    &&\text{(with $g_{\GW}$ and $f_{\GW}$ $\CC$-valued and $f_{\GW}$ invertible)}
    \end{align*}
the A-model of the quintic (\emph{i.e.} $J_{\GW}$),
matches the B-model of the quintic (\emph{i.e.} $I_{\GW}$),
via a mirror map
    \begin{align}\label{eq:mirrormap}
    &\frac{I_{\GW}(q,z)}{f_{\GW}(q)}=J_{\GW}\left(q',z\right).
\end{align}
The above explicit relation between $I_{\GW}$ and $J_{\GW}$
determines $J_{\GW}$ only on the one-parameter space $\CC \fie_1$;
however, this determines entirely the $J$-function above via
tautological relations as discussed in Remark \ref{rem:CYsimpl}
and in the proof of Corollary \ref{cor:sympl}.

We provide the same picture on the LG side (a direct consequence of
Thm.~\ref{thm:FJRW-I-funct},
    and Rem.~\ref{rem:deducing_main}-2).
\begin{thm}\label{thm:PF}
    Consider the $H_{\RW}((z^{-1}))$-valued function (where
    $[a]_n=a(a+1)\dots(a+n-1)$)
    $$I_{\RW}(t,z)= z\sum_{k=1,2,3,4} \frac{1}{\Gamma(k)}\sum_{l\ge 0} \frac{{([\frac{k}{5}]_l)^5}
    \ {t^{k+5l}}}{{[k]_{5l}}\ z^{k-1}}\, \phi_{k-1},$$
    whose four summands span the solution space of
    the Picard--Fuchs equation
    \begin{align*}
    &\left[{D_t}^4-5^5t^{-5}\prod_{m=1}^4(D_t-mz)\right]I_{\RW}=0&&
    \text{\Bigg(for $D_t=zt\frac{\partial}{\partial t}\Bigg)$}
    \end{align*}
    and coincide with the period integrals
    at the Gepner point computed by Huang, Klemm, and Quackenbush \cite{HKQ}.
    The above $I$-function and the $J$-function of $\RW$-theory
    are related by an explicit change of variables (mirror map)
    \begin{align*}
    &t'=\frac{g_{\RW}(t)}{f_{\RW}(t)}
    &&\text{(with $g_{\RW}$ and $f_{\RW}$ $\CC$-valued and $f_{\RW}$ invertible)}
    \end{align*}
    satisfying
    \begin{align*}
    &\frac{I_{\RW}(t,z)}{f_{\RW}(t)}=J_{\RW}(t',z).
    \end{align*}
\end{thm}

     The Picard--Fuchs equation
    in the above statement coincides with that of the quintic three-fold
    for $q=t^{-5}$.
    After the identification $q=t^{-5}$
    of the coordinate patch at $t=0$
    with the coordinate patch at $q=\infty$,
    the two $I$-functions are solutions of the same
    Picard--Fuchs equation.
    Since $I_{\GW}$ and $I_{\RW}$ take values in two isomorphic
    state spaces, we can compute the analytic continuation of $I_{\GW}$
    and obtain two different bases spanning the space of solutions of
    the same Picard--Fuchs
    equation. Therefore, in Section 4, we have the following corollary.

    \medskip

    \noindent\textbf{Corollary \ref{cor:sympl}.} \emph{
    There is a $\CC[z,z^{-1}]$-valued degree-preserving symplectic transformation $\UU$
    mapping $I_{\RW}$ to the analytic continuation of $I_{\GW}$
    near $t=0$. This proves the genus-zero LG/CY correspondence
    (Conjecture \ref{conj},(1)).}

    \medskip

    We have explicitly computed $\UU$ using the Mellin--Barnes method
    for analytic
    continuation, \eqref{eq:FJRWtoGW}.
    In this way, in the case of the quintic three-fold,
    the higher genus LG/CY correspondence
    assumes an explicit form.
    Indeed, in the terms of the second part of
    Conjecture \ref{conj},
    the quantization
    $\widehat{\UU}$ is a differential operator
    which we expect to yield the full higher genus
    Gromov--Witten partition function
    when applied to
    the full higher genus Fan--Jarvis--Ruan--Witten partition function.

\subsection{Organization of the paper}

The rest of the paper is organized in three sections: on $\RW$
theory (Sect.~2), on the LG/CY conjecture (Sect.~3), and on its
proof in genus zero (Sect.~4). In the appendix, we recall known facts
about orbifold curves used in section 2.

\subsection{Acknowledgements}
The first author would like to thank Albrecht Klemm
and Hsian-Hua Tseng for
helpful conversations and \'Etienne Mann and Thierry Mignon
for their remarks on an early version of this paper. A special thank is due to
Dimitri Zvonkine for the collaboration on Givental's quantization \cite{CZ},
which opened the way to the present work.

The second author would like to thank Tom
Coates, Hiroshi Iritani, and Yuan-Pin Lee for many fruitful
conversations on Givental formalism. A special thank goes to Tom
Coates for the collaboration on crepant resolutions \cite{CR}, which
gave many important insights to LG/CY correspondence.
Special
thanks also go to Huijun Fan and Tyler Jarvis for the collaborations
this work is based on.

We would like to thank the referees for their
comments which led to a clearly improved version of this paper.

\section{Fan--Jarvis--Ruan--Witten theory}\label{sect:FJRW}
    We recall the theory for the Fermat quintic polynomial in
    three steps: the state space, the moduli space, and the
    enumerative geometry.

\subsection{The state space associated to the Fermat quintic singularity}\label{sect:FJRWsing}
We treat the
singularity  $\{(x_1,\dots, x_5)\in \CC^5\mid
W(x_1,\dots,x_5)=0\}$ where $W$ is the Fermat quintic polynomial
in five variables
\begin{equation}\label{eq:Wexplicit}
W(x_1,\dots,x_5)=x_1^5+x_2^5+x_3^5+x_4^5+x_5^5.
\end{equation}
We consider a natural symmetry group
attached to $W$ (and to any other quasihomogenous
polynomial, see \cite{FJR1}):
the group of diagonal symmetries $G_W$; we have
$$G_W=\pmmu_{5}\times
\pmmu_{5}\times
\pmmu_{5}\times
\pmmu_{5}\times
\pmmu_{5}\subset U(1)^5,$$
which contains the cyclic order-$5$ subgroup generated by
$$J_W:=\left(\exp({2\pi {\cxi}{1}/{5}}), \dots, \exp({2\pi
\cxi{1}/{5}})\right)\in (\CC^\times)^5$$
(we omit the subscript $W$ if no ambiguity can arise)\footnote{In the general theory, applying to any quasihomogeneous polynomial,
$G_W$ is a finite abelian subgroup of $U(1)^N$
and
the exponents defining $J$ are given by $2\pi \cxi$ times
the charges $q_1,\dots, q_N$
of the variables satisfying
$\la W(x_1,\dots,x_N)=W(\la^{q_1}x_1,\dots,\la^{q_N}x_N)$.)}.

FJRW theory is an analogue of Gromov--Witten theory attached to
a pair $(W, H)$,
where $H\subseteq G_W$ contains $J$.
The state space is a bigraded vector space of the form
$\Hcal_{\RW}(W,H)=\bigoplus_{\ga\in H} \Hcal_\ga.$
We will define the summands for
the Fermat quintic polynomial and for the group $\langle J\rangle$;
the general setup may be found in \cite{FJR1}.
Let $(\CC^5)_\ga$ denote the linear space
which is fixed by $\ga$, and let $N_\ga$ be its dimension.
Let $W_\ga$ be $W\rest{(\CC^5)_\ga}$ and
let
$W^{+\infty}_\ga$ denote $({\Re} W_{\gamma})^{-1}
\left(]\rho, +\infty[\right)$ for $\rho \gg 0$.
We have
\begin{equation}\label{eq:sspace}
\Hcal_{\RW}(W,\langle J\rangle )=\bigoplus_{\ga\in \langle J\rangle}
\Hcal_\ga, \qquad \text{where}\qquad
\Hcal_\ga=H^{N_\ga}((\CC^5)_\ga,W_\ga^{+\infty};\CC)^{\langle J\rangle},
\end{equation}
and $H^*(\ \  )^{\langle J\rangle}$
stands for the subspace of $\langle J\rangle$-invariant elements
(we may regard the natural $\langle J\rangle$-action
as a monodromy action).
As in Chen--Ruan cohomology, we should introduce a twist
with respect to the standard bidgegree of the forms in $\Hcal$.
Recall that a diagonal action $\mu$ on $\CC^M$ of eigenvalues
$\exp(2\pi\cxi h_1), \dots, \exp(2\pi\cxi h_M)$ (with
$h_j \in [0,1[$) is said to be of \emph{age}
$$\age(\mu)=\textsum_{1\le j\le M} h_j\in \QQ.$$

For a class $\alpha \in \Hcal_{\gamma}$, we shift the bidegree by
$(\age(\ga)-1,\age(\ga)-1)$.
In this way the total degree is
\begin{equation}\label{eq:degshift}
\deg_W(\alpha)=\deg(\alpha)+ 2\age(\ga)-2.\end{equation}
The state space $\Hcal_{\RW}(W,\langle J\rangle)$ carries a natural Poincar\'e pairing
$(\ ,\ )_{\RW}$, see \cite{FJR1}.

\begin{notn} We will usually
write $\mathcal H_{\RW}(W)$ omitting the group,
because we usually refer to the
state space of $W$ with
respect to $\langle J{}\rangle$.
\end{notn}

\begin{rem}\label{rem:conditionCY} In the
geometric LG/CY correspondence, we consider a CY
hypersurface
$X_W=\{W=0\}$ inside the weighted projective stack
$\PP(c_1,\dots,c_N)$.
The state space on the CY-side is the Chen--Ruan cohomology of the stack
$X_W$.
In general, there exists a degree-preserving
  vector space isomorphism between
    $\mathcal H_{\RW}(W)$ and the Chen--Ruan cohomology of the
    weighted projective CY hypersurface of equation $W=0$, see \cite{ChiR}.
In this paper we only need the case of the
quintic polynomial, which we analyze in the following example.
\end{rem}

    \begin{exa}\label{exa:isosspace}
    For the Fermat quintic polynomial and the cyclic group
    $G=\langle J{}\rangle$ of order 5, let us compute,
    for each element $J^{m}=(e^{2\pi\cxi m/5},\dots,e^{2\pi\cxi m/5})$
    with $m=0,\dots,4$,  the space
    $\Hcal_{J^{m}}$ and the degree of its elements.
    The degree is even if and only if $m$ does not vanish.

    Let $m\ne 0$. The index
    $\age(e^{2\pi\cxi m/5},\dots,e^{2\pi\cxi m/5})$ equals $m$.
    Therefore, for $\gamma=J^{m}$
    all the elements of $\Hcal_\gamma$ have degree
    $2m-2$.
    The group $\Hcal_\gamma$ in all these cases
    has a single generator because $\CC^5_\gamma$ equals $\{0\}$; we denote by
    $e(J^{m})\in \Hcal_{J^{m}}$ its generator, which can be
    regarded as the constant
    function $1$ on $\CC^5_\gamma$. In this way we obtain four elements
    of degree $0,2,4$ and $6$,
    which correspond to
    the generators of $H^0(X_W;\CC)$,
    $H^2(X_W;\CC)$, $H^4(X_W;\CC)$, and $H^6(X_W;\CC)$.

    Let $m=0$. We have $\Hcal_{J^0}=H^5(\CC^5, W^{+\infty};\CC)^{\langle J\rangle}$,
    which is isomorphic
    to the degree-$3$ cohomology group of $X_W$ (by
    the theory of the Milnor fibre, see for instance \cite{ChiR}).
    By definition,
    the total degree of the elements of $\Hcal_{J^0}$ is $3$.
\begin{equation}\label{eq:localindicesforWG}
\begin{tabular}{c||c|c|c|c|c|}
   & $ \forall \al \in \Hcal_{J^0}$ &$ e(J)\in \Hcal_{J}$ &$e(J^2)\in \Hcal_{J^2}$ & $e(J^3)\in \Hcal_{J^3} $ & $e(J^4)\in \Hcal_{J^4} $
   \\
   \hline
$\deg$ & 3  &
0 & 2 & 4 & 6\\
\end{tabular}\ .
\end{equation}
    Therefore, we have a degree-preserving vector space isomorphism
    \begin{equation}\label{eq:statespaceiso}
    \Hcal_{\RW}(W)\cong H^*(X_W; \CC)\end{equation}
    compatible with Poincar\'e pairings on both sides.
    The left hand side does not have any natural multiplication in the
   classical sense. FJRW theory builds upon
    the above isomorphism \eqref{eq:statespaceiso}
    and provides a quantum multiplication which turns
    the vector space isomorphism into a quantum ring isomorphism.
\end{exa}

\subsection{The moduli space associated to the Fermat quintic singularity}
FJRW theory is a theory of curves equipped with roots of the
log-canonical line bundle and its powers. The study of moduli of roots requires
the notion of orbifold curves (Appendix \ref{appendix} summarizes the standard results
on this type of curves needed in this paper).

\paragraph{The moduli stack.} In $\RW$ theory,
a moduli functor is
attached to any quasihomogenous nondegenerate polynomial $W$ as mentioned in Section \ref{sect:mainres}.
In this paper we focus on the Fermat quintic polynomial and the setup is much simpler.
The moduli functor classifies $5$-stable curves (\emph{i.e.} orbifold curves
with order five cyclic stabilizers at markings and  nodes, see Definition \ref{defn:dstable}), equipped with five fifth roots of
the log-canonical line bundle: $L_1,
L_2,L_3,L_4,L_5$ satisfying $(L_j)^{\otimes 5}\cong \omega_{\log}$ all $j$.
We write $W_{g,n,G_W}$ for this moduli functor; we have
$$W_{g,n,G_W}=\RRR_5\times_{\MMM_{g,n,5}}\RRR_5
\times_{\MMM_{g,n,5}}\RRR_5\times_{\MMM_{g,n,5}}\RRR_5\times_{\MMM_{g,n,5}}\RRR_5,$$
\emph{i.e.} an \'etale cover of the moduli stack of $5$-stable curves $\MMM_{g,n,5}$ (the stack of $5$th roots $\RRR_5$
of $\omega_{\log}$ is defined in the most natural way: $5$-stable curves with a line bundle
whose $5$th tensor power is isomorphic to $\omega_{\log}$---see Appendix \ref{appendix}, \eqref{eq:droots}).

\paragraph{The marking's type.}
The definition of
$W_{g,n,G_W}$ as the fibred product of
five copies of $\RRR_{5}$,
allows us to generalize the notion of
type of a marking already defined in the literature for
markings of spin curves,
\emph{i.e.} for the
stack $\RRR_5$.
The stack
$\RRR_{5}$ is the disjoint union
\begin{equation}\label{eq:spincompo}
\RRR_5=
\bigsqcup_{0\le \Theta_1,\dots, \Theta_n<1}
\RRR_5(e^{2\pi \cxi\Theta_1},\dots,e^{2\pi\cxi\Theta_n})),
\end{equation}
where \begin{multline*}
\RRR_5(e^{2\pi \cxi\Theta_1},\dots,e^{2\pi\cxi\Theta_n})=
\\\bigg\{(C;\sigma_1,\dots,\sigma_n; L; \fie)\mid \fie\colon
L^{\otimes 5}\xrightarrow{\ \sim \ } (\omega_{\log})
 \text{\ \  and \ \ $\Theta_i=\mult_{\sigma_i} L\in [0,1[ \ \ \  \forall i$}\bigg\}.\end{multline*}
The index $\mult_{\sigma_i}L$ is determined by the local indices of the
universal $5$th root $L$
at the $i$th marking $\sigma_i$.
More explicitly, the local picture of
$L$ over $C$ at the $i$th marking
$\sigma_i$ is parametrized by the pairs $(x,\la)\in \CC^2$, where
$x$ varies along the curve and $\la$ varies along the fibre of the
line bundle.
The stabilizer $\pmmu_{5}$
at the marking acts as $(x,\la)\mapsto (\exp({2\pi \cxi/5})
x, \exp({2\pi \cxi\Theta_i})\la)$.
In this way, the local picture provides an explicit definition of $\Theta_1,\dots,\Theta_n$.

As a straightforward consequence,
we get a decomposition of the stack $W_{g,n,G_W}$
into several connected components obtained
via  pullback of
the fibred products of the components of type \eqref{eq:spincompo}
for the moduli stacks $\RRR_{5}$. We
restate this decomposition in the following definition.
\begin{defn}
Fix $n$ multiindices
with five entries $\gamma_i=(e^{2\pi\cxi \Theta^i_1},\dots,e^{2\pi\cxi \Theta^i_5})
\in U(1)^5$ for
$i$ equal to $1,\dots, n$, $j=1,\dots,5$, and $\Theta^i_j\in [0,1[$.
Then
$W(\gamma_1,\dots,\gamma_n)_{g,n,G_W}$
is the stack of objects of $W_{g,n,G_W}$ satisfying
the relation
$\Theta^{i}_{j}=\mult_{\sigma_i} L_j,$
where $\Theta^{i}_{j}$ is the $j$th entry of
$\gamma_i$.
\end{defn}
\begin{pro}\label{pro:decomp}Let $n>0$.
The stack $W_{g,n,\langle J \rangle}$ is the disjoint union
$$W_{g,n,G_W}=
\bigsqcup_{\gamma_1,\dots,\gamma_n \in U(1)^5}
W(\gamma_1,\dots,\gamma_n)_{g,n,G_W}.$$
The stack
$W(\gamma_1,\dots,\gamma_n)_{g,n,G_W}$ is nonempty if and only if
\begin{equation}\label{eq:degcond}\begin{cases}
\gamma_i=(e^{2\pi\cxi \Theta^i_1},\dots,e^{2\pi\cxi \Theta^i_5})\in G_W&i=1,\dots,n;\\
\frac15(2g-2+n)-\sum_{i=1}^n \Theta^i_j\in \ZZ&j=1,\dots,5.
\end{cases}\end{equation}
As a consequence,
its degree as an \'etale cover of the stack of $5$-stable curves equals $\abs{G_W}^{2g}/5^5$.
\qed\end{pro}

\begin{proof} The proof is an immediate consequence of
\cite[Rem.~2.2.14]{FJR1} and of the torsor structure
of the category of $5$th roots, see \cite{chstab}.
\end{proof}

\begin{notn}[markings of Ramond and Neveu--Schwartz type]
All the objects of the moduli stack
     $W(\gamma_1, \cdots,
     \gamma_n)_{g,n,G_W}$ have $n$ markings whose local
     indices are $\gamma_1,\dots, \gamma_n$. This allows us
     to introduce the notion of marking of Neveu--Schwarz (NS) type
     and Ramond (R) type.
     $$\text{A marking with local index $\gamma=(e^{2\pi\cxi \Theta_1},\dots,e^{2\pi\cxi \Theta_N})$ is }
     \begin{cases}\text{NS} & \text{if $e^{2\pi\cxi \Theta_j}\ne 1 \ \forall j$}\qquad \ (\Leftrightarrow (\CC^5)_\gamma=\{0\}),\\
     \text{R} & \text{otherwise.}\end{cases}$$
\end{notn}

\begin{exa}
We consider for example
$3$-pointed genus-$0$
curves equipped with five $5$th roots of $\omega_{\log}\cong \Ocal(-2+\text{$3$ markings})\cong\Ocal(1)$.
Let us consider a case where all these roots are isomorphic to each other
and have local index $\xi_5^{2}$ at each marking.
Note that this case occurs by Proposition \ref{pro:decomp} and that each marking is of NS type.
Similarly, the stack of
$5$-pointed genus-$0$
curves equipped with five isomorphic $5$th roots of $\omega_{\log}\cong\Ocal(-2+\text{$5$ markings})=\Ocal(3)$
of type $\xi_5^0=1$ at the first marking and $\xi_5^{2}$ at the four
remaining markings
is nonempty by equation \eqref{eq:degcond} of Proposition
\ref{pro:decomp}. Here, the first marking is of R type whereas the remaining
markings are NS as before.
\end{exa}

\paragraph{The substack with respect to a group $H$.}
In perfect analogy with the state space $\Hcal_{\RW}(W,H)$ of
the previous subsection, the
moduli stacks are attached to each pair $(W,H)$ with $H$ containing $J=\xi_5\mathbb I$
and included in $G_W$. We refer to
\cite{FJR1} for a general definition of
$W_{g,n,H}$, an open and closed substack of $W_{g,n,G_W}$.
In this paper, we focus on the case where $H=\langle J\rangle$, where
$W$ is identified as the image of the natural map
\begin{align*}
\RRR_5&\to W_{g,n,G_W}\\
(C;L)&\mapsto (C;(L,L,L,L,L)).
\end{align*}
The stack $W_{g,n,\langle J\rangle}$
fits in the following diagram
\begin{equation}\label{eq:corelocus}
\RRR_5 \ \overset{\ }\twoheadrightarrow \
W_{g,n,\langle J\rangle}\  \overset{\ }{\hookrightarrow}\ W_{g,n,G_W}.
\end{equation}The surjective morphism $\RRR_5\to W_{g,n,\langle J\rangle}$ is
locally isomorphic to $B\mu_5\to B{(\pmmu_{5})^5}$.

\begin{rem}
As in Proposition \ref{pro:decomp}, for $n>0$, we have
     $$W_{g,n,\langle J\rangle}=
     \bigsqcup_{\gamma_1,\dots,\gamma_n\in \langle J\rangle}
     W(\gamma_1,\dots,
     \gamma_n)_{g,n,\langle J\rangle},$$
     where $\gamma_i\in \langle J\rangle$ is the local index at the
     $i$th marked point.
\end{rem}

\subsection{The enumerative geometry of the Fermat quintic singularity}\label{sect:enum}
In this section we develop the enumerative geometry of $W_{g,n,\langle J\rangle}$:
we consider tautological classes and we focus on the $\RW$ invariants.

\paragraph{Tautological classes.}
We recall that the moduli stacks
$W_{g,n,\langle J\rangle}$ are equipped with
the classes $\psi_i$:
$$\psi_i\in H^2(W_{g,n,\langle J\rangle})\qquad \text{for $i=1,\dots, n$}.$$
There are many ways to define these classes, the most
straightforward (following \cite[\S8.3]{AGV})
is to consider the first Chern class of the
bundle whose fibre at a point is the
cotangent line to the corresponding \emph{coarse} curve at the $i$th marked point.
We can also define the kappa classes; simply set
$$\kappa_h=\pi_*(c_1(\omega_{\log})^{h+1})\in H^{2h}(W_{g,n,\langle J\rangle})\qquad \text{for $h\ge 0$},$$
where $\pi$ is the universal curve $\Ccal_{g,n}\to W_{g,n,\langle J\rangle}$.
If we identify each stack $W_{g,n+1}(\gamma_1,\dots,\gamma_n,1)$
to the universal curve $\pi \colon \Ccal \to W_{g,n,\langle J\rangle}(\gamma_1,\dots,\gamma_n)$,
then we can
express $\kappa_h$ as $\pi_*(\psi_{n+1}^{h+1})$.

Let us write the degree-$2h$ term of the Chern character as
$$\ch_h(R\pi_*\mathcal L)\in H^{2h}(W_{g,n,\langle J\rangle}),$$
where $\Lcal$ is the universal $5$th root of the log-canonical bundle of the universal
curve of $W_{g,n,\langle J\rangle}$.
Indeed, we can express this cohomological class
in terms of psi classes and kappa classes.

The statement requires further analysis of the boundary locus $\partial W_{g,n,\langle J\rangle}$,
the substack of $W_{g,n,\langle J\rangle}$ representing singular curves.
The normalization $N(\partial W_{g,n,\langle J\rangle})$
of this locus can be regarded as the stack
parametrizing pairs (objects of $W_{g,n,\langle J\rangle}$, nodes)
in the universal curve.
Finally, we consider the double \'etale cover
$\Scal$ of this normalization; namely,
the moduli space of triples of the form (objects of
$W_{g,n,\langle J\rangle}$, node, a branch of the node).
The stack $\Scal$ is naturally equipped with two line bundles
whose fibres are the cotangent lines to the chosen branch
of the \emph{coarse} curve. We write
$$\psi, \psi'\in H^2(\Scal)$$
for the respective first Chern classes.
Note that in this notation
we privilege the coarse curve because in this way the classes
$\psi, \psi'$ are more easily related to the classes $\psi_i$ introduced above.
Recall, see for example \cite[\S2.2]{CZ}, that
a $5$th root of $\omega_{\log}$ at a node of a $5$-stable
curve determines local indices $a, b\in [0,1[$
such that $a+b\in \ZZ$ in one-to-one correspondence with
the branches of the node. In this way on $\Scal$,
which naturally maps to $W_{g,n,\langle J\rangle}$,
the local index attached to the chosen branch determines
a natural
decomposition into open and closed substacks
and natural restriction morphisms
\begin{equation*}\label{eq:decompsing1}
\Scal=\bigsqcup_{\Theta=0,1/5,\dots, 4/5}
\Scal_\Theta,\qquad
\qquad j_\Theta\colon \Scal_\Theta \to W_{g,n,\langle J\rangle}.
\end{equation*}
\begin{pro}\label{pro:GRR}
Let $\ch_h$ be the degree-$2h$ term of
the restriction of the Chern character  to
the stack
$W(\gamma_1,\dots, \gamma_n)_{g,n,\langle J\rangle}$,
where $\gamma_i=(e^{2\pi\cxi \Theta^i_{1}},\dots,
e^{2\pi\cxi \Theta^i_{5}})$ for $\Theta^i_j\in [0,1[^5$.
We have
$$\ch_h(R\pi_* \mathcal L)=\frac{B_{h+1}
(\frac15)}{(h+1)!}\kappa_h-
\sum_{i=1}^n
\frac{B_{h+1}({\Theta^i_{j}})}{(h+1)!}\psi_i^h+
\frac{5}{2} \sum _{0\le \Theta<1}
\frac{B_{h+1}(\Theta)}
{(h+1)!} (j_\Theta)_* \left(\sum_{a+a'=h-1} \psi^a (-\psi')^{a'}\right).$$
\end{pro}
\begin{proof}
This is an immediate consequence of
the main result of \cite{ch}.
\end{proof}

\paragraph{The FJRW invariants
of $(x_1^5+x_2^5+x_3^5+x_4^5+x_5^5)$.}
     The $\RW$ invariants are defined in \cite{FJR1}
     for any choice of nonnegative
     integers $a_1,\dots, a_n$ (associated to powers of psi classes
     $\psi_1^{a_1}, \dots, \psi_1^{a_n}$)
     and any choice of elements $\al_1,\dots,\al_n\in \Hcal_{\RW}(W,H)$; we
     will simply write
     $$\langle \tau_{a_1}({\al_1}),\dots,\tau_{a_n}({\al_n})
     \rangle^{\RW}_{g,n}$$
     omitting $W$ and $H$, because from now on
     we only consider the quintic polynomial $W$ and the group $H=\langle J\rangle$.
     As in $\GW$ theory the definition of the invariants is
     given by taking the cap product of a
     virtual fundamental cycle with the above mentioned cohomology classes.
     In $\RW$ theory, the virtual cycle is
     $[W(\gamma_1,\dots,\gamma_n)]^{\vir}_{g,n,H}$
     (with $\gamma_i$ such that  $\al_i\in \Hcal_{\gamma_i}$),
     defined on the moduli stack $W(\gamma_1,\dots,\gamma_n)_{g,n,H}$.
     (We simply write $[W]_{g,n,H}^{\vir}$ when we refer to
     the entire stack $W_{g,n,H}$.)

     For any genus $g\ge 0$, this virtual cycle is first capped with
     $\al_1,\dots,\al_n\in \Hcal_{\RW}(W,H)$, yielding\footnote{See formula of the degree ``$r$'' in \cite[Defn.~4.1.5.(1)]{FJR1}; notice that ``$\hat c$''
     equals $3$ for the quintic polynomial by \cite[eq.(42)]{FJR1}} a cycle of
     real dimension
     $2n-\sum_j \deg(\al_j)$,
     and then with powers of
     psi classes $\psi_1^{a_1},\dots, \psi_n^{a_n}$, yielding a zero-dimensional
     cycle if and only if
     \begin{equation}\label{eq:dimensioncompatibility}
     2n-\sum_j \deg(\al_j)-2\sum_j{a_j}=0.\end{equation}
     By the string equation (see \cite[Thm.~4.2.8]{FJR1}) each
     invariant containing one entry $\tau_a(\al)$ with $a=0$ and
     $\al=e(J)$ (see notation of Example \ref{exa:isosspace}) can be expressed in terms
     of the remaining invariants unless it is of the form
     $\langle \tau_{0}(e(J)),\tau_{0}({\al_2}),\dots,\tau_{0}({\al_n})
     \rangle^{\RW}_{g,n}$.
     The latter invariant  vanishes for $(g,n)\neq(0,3)$ and equals $(\al_2,\al_3)_{\RW}$ otherwise
     (see \cite[Thm.~4.2.2.(C4b)]{FJR1}).
     This completely reduces the problem to computing $\RW$
     invariants $\langle \tau_{a_1}({\al_1}),\dots,\tau_{a_n}({\al_n})
     \rangle^{\RW}_{g,n}$ with $\deg(\al_j)\ge 2$ or $a_j\ge 1$ for all $j$.
     Furthermore, \eqref{eq:dimensioncompatibility} implies that
     either $\al=e(J^2)$ and $a=0$, or $\al=e(J)$ and $a=1$.

     \begin{rem}\label{rem:eveninsertions}
     The discussion above yields the following
     two properties for the generating function of all
     $\RW$ invariants of the form $\langle \tau_{a_1}({\al_1}),\dots,\tau_{a_n}({\al_n})
     \rangle^{\RW}_{g,n}$ with $\deg(\al_1), \dots,\deg(\al_n)\in 2\ZZ$ (\emph{i.e.}
     $\al_1,\dots,\al_n\in \Hcal^{\rm ev}_{\RW}(W)$).
     First, this class of invariants spans the entire $\RW$ theory and, second, it
     embodies a self-contained theory.
     More precisely, we have the following statements.
     \begin{enumerate}
     \item The invariants with only even degree insertions span
     the entire $\RW$ theory. This means that
     the state space $\Hcal_{\RW}(W)$ (equipped with its nondegenerate inner Poincar\'e pairing)
     and the set of invariants with only even degree insertions allow us to express
     all the remaining invariants via the string equation.
     \item Furthermore these invariants can be singled out and form a self-contained theory.
     This means that the restriction of the $J$-function from \eqref{eq:JforRW}
     $$J_{\RW}\colon t\mapsto
     J_{\RW}(t,z)\in \Hcal_{\RW}(W)((z^{-1}))$$
     to the even degree entries in $\Hcal^{\rm ev}_{\RW}((z^{-1}))$
     is equal to the $J$-function defined by the same expression \eqref{eq:JforRW}
     with $\phi_\epsilon$ running \emph{only} in $\Hcal^{\rm ev}_{\RW}(W)$.
     This amounts to say that, in $\RW$ theory,
     if all entries but one are of the form $\tau_0(\al)$ with
     $\al\in \Hcal^{\rm ev}_{\RW}(W)$, then the remaining entry $\tau_k(\phi_\epsilon)$ necessarily
     satisfies $\phi_\epsilon\in \Hcal^{\ev}_{\RW}(W)$ or yields a vanishing invariant.
     This claim follows immediately from \eqref{eq:dimensioncompatibility}.
     \end{enumerate}
     This explains why in the next section we will deal
     with
     $J$-functions and Lagrangian cones only for even degree insertions.
     Indeed, a parallel argument, also based on a dimension count, holds for $\GW$
     theory of the quintic three-fold; see Remark \ref{rem:CYsimpl}.
     \end{rem}

     \begin{rem}\label{rem:concavity}
     For genus-zero invariants,
     the case where all insertions $\tau_a(\al)$ satisfy  $\deg(\al)\in 2\ZZ$
     (or, in other words, the case
     where all  insertions are  of NS type)
     can be further detailed: for $m_1,\dots,m_n\in \{1,\dots,4\}$,
     by \cite[Defn.~4.1.1.(5a)]{FJR1},
     the virtual cycle $[W(J^{m_1},\dots,J^{m_n})]^{\vir}_{g,n,G}$
     is Poincar\'e dual to $c_{\rm top}((R^1\pi_*\Lcal)^\vee)^5$. More precisely,
     let $\Lcal$ be the universal line bundle on the universal curve of
     $(A_4)_{0,n}=\RRR_{5}$ mapping
     to $W_{0,n,\langle J\rangle}$. The notation $(A_4)_{g,n}$ is a special case of
     the notation $W_{g,n,\langle J_W\rangle}$ with $W=x^5$, \emph{i.e.} the polynomial attached to the $A_4$-singularity;
     we omit the group in the notation, because there is no ambiguity for $A_4$-singularities.
     Then, for $m_1,\dots,m_n\in \{1,\dots,4\}$, we have
     \begin{equation}\label{eq:concavity}
     \langle \tau_{a_1}({\phi_{m_1-1}}),\dots,\tau_{a_n}({\phi_{m_n-1}})
     \rangle^{\RW}_{0,n}= 5 \left(\prod_{i=1}^n \psi_i^{a_i} \cup
     c_{\rm top}((R^1\pi_*\Lcal)^\vee)^5 \right)\cap [A_4(J^{m_1},\dots, J^{m_n})]_{0,n}.
     \end{equation}
     The fact that the above higher direct image is a vector bundle
     is well known (see \cite{Wi2}), but we recall it for clarity.
     It is enough to show that
     $\Lcal$ has no global sections on every geometric fiber of $\pi$. If the fibre is
     smooth (hence reducible) this follows from the fact that the pushforward
     $p_*\Lcal$ via the map forgetting the stabilizers at the markings
     has negative degree
     on the fibre.
     For a reducible fibre $C$ this happens
     because
     on each irreducible component $Z$
     the degree $d_Z$ of $p_*\Lcal$ is less than the number of points
     meeting the rest of the fibre minus $1$:
     $$d_Z< \#(Z\cap \ol{C\setminus Z})-1.$$
     Notice that this holds only
     because we have $m_1,\dots, m_n>0$ and $g=0$.
     We finally point out that the
     virtual cycle above is
     simply the
     self-intersection of
     five copies of the $5$-spin virtual cycle
     $[A_4(J^m_1,\dots,J^{m_n})]^{\vir}_{g,n}$ in $A_4(J^m_1,\dots,J^{m_n})_{g,n}$.
     \end{rem}

 \section{LG/CY-correspondence for the quintic three-fold}\label{sect:conj}
    The $\RW$ theory of the quintic singularity is definitely different from
    the GW theory of the quintic three-fold, but it also fits in Givental's formalism.
    We recall the construction and we state the LG/CY conjecture for the quintic
    three-fold.

\subsection{Givental's formalism for GW and FJRW theory}\label{sect:setup}
The genus-zero invariants of both theories are encoded
by two Lagrangian cones, $\Lcal_{\GW}$ and $\Lcal_{\RW}$,
inside two symplectic vector spaces, $(\Vcal_{\GW}, \Omega_{\GW})$ and
$(\Vcal_{\RW},\Omega_{\RW})$.
The two symplectic vector spaces also allow us to state the conjectural
correspondence  in higher genera.
By Remark \ref{rem:eveninsertions},
we focus on the even degree part both in
$\GW$ theory and in the $\RW$ theory.
We recall the two settings simultaneously by using the
subscript $\W$, which can be read as ${\GW}$ or ${\RW}$.

\paragraph{The symplectic vector spaces.}
We define the vector space $\Vcal_{\W}$ and
its symplectic
form $\Omega_{\W}$.

The elements of the vector space $\Vcal_{\W}$
are the Laurent series with values
in a state space $H_{\W}$
$$\Vcal_{\W}=H_{\W}\otimes \CC((z^{-1})).$$
In $\RW$ theory the state space is normally the entire
space $\Hcal_{\RW}(W)$; however, as already mentioned,
    we can focus on even-degree entries
    of $\Hcal_{\RW}(W)$ and $H^*(X_W;\CC)$
    (see Remark \eqref{rem:eveninsertions}).
Therefore, in view of the computation of
the $\RW$ potential we can
restrict to the even-degree part
$$H_{\RW}=\Hcal^{\ev}_{\RW}(W)=
\bigoplus_{\iii=0}^3 e(J^{\iii+1}) \CC,$$ with nondegenerate inner
product $(\  \,,  \ )_{\RW}$ (recall
from equation \eqref{eq:localindicesforWG} that the only element of
odd degree are those of  $\Hcal_{J^0}$).
Similarly, in $\GW$ theory, the state space $H_{\W}$ will be the
even degree cohomology ring of the variety $X_W$, which is
generated by the hyperplane section $H$:
$$H_{\GW}=H^{\ev}(X_W; \CC)=\bigoplus_{\iii=0}^3 [H^\iii] \CC,$$
with a natural nondegenerate inner product
induced by Poincar\'e duality; we denote it by $(\ , \ )_{\GW}$.
We express the  basis of $H_{\W}$ as
$\Phi_0,\dots, \Phi_3$ and the dual basis $\Phi^0,\dots, \Phi^3$ as
$$\Phi_\iii=\begin{cases}\varphi_\iii=[H^\iii] & \text{in GW theory,}\\
\\ \phi_\iii=e(J^{\iii+1}) & \text{in FJRW theory},
\end{cases}\qquad \Phi^{\iii}=\begin{cases}\varphi^\iii=5^{-1}{[H^{3-\iii}]} & \text{in GW theory,}\\
\\ \phi^\iii=e(J^{4-\iii}) & \text{in FJRW theory}.
\end{cases}$$
Clearly $(\Phi_\iii,\Phi^k)_{\W}=\delta_{h,k}$.
\begin{conv}\label{conv:mod5}
When we write $\Phi_k$  or $\Phi^k$ for an
integer $k$ lying outside the above range, we mean $\Phi_\iii$ for $\iii={5\{ \frac k5\}}$,
where $\{\ \}$ is the fractional part.
\end{conv}

The vector space $\Vcal_{\W}$ is equipped with the symplectic form
$$\Omega_{\W}(f_1,f_2)=\Res _{z=0} (f_1(-z),f_2(z))_{\W}.$$
In this way $\Vcal_{\W}$ is polarized as
$\Vcal_{\W}=\Vcal^+_{\W}\oplus \Vcal^-_{\W}$, with
$\Vcal^+_{\W}=H_{\W}\otimes \CC[z]$ and
$\Vcal^-_{\W}=z^{-1}H_{\W}\otimes \CC[[z^{-1}]]$,
and can be regarded as the total cotangent space of
$\Vcal^+_{\W}$. The points of $\Vcal_{\W}$ are parametrized by
Darboux coordinates
$\{q^\iii_{a}, p_{l,j}\}$ and can be written as
$$\sum_{a\ge 0} \sum _{\iii=0}^{3} q^\iii_{a} \Phi_\iii z^a+
\sum _{l\ge 0} \sum_{j=0}^3 p_{l,j} \Phi^j (-z)^{-1-l}.$$

\paragraph{The potentials.}\label{sect:lagcone}
We review the definitions of the
potentials
encoding the invariants of the two theories.

In FJRW theory, the invariants are the intersection numbers
\begin{equation}\label{eq:FJRWinv}
\langle \tau_{a_1}(\phi_{\iii_1}),\dots,\tau_{a_n}(\phi_{\iii_n})
\rangle^{\RW}_{g,n}= \frac{1}{5^{g-1}} \prod_{i=1}^n \psi_i^{a_i}
\cap
\left[W\left(J^{\iii_1+1},\dots,J^{\iii_n+1}\right)\right]_{g,n,\langle J\rangle}^{\vir},
\end{equation}
see Section \ref{sect:enum} and equation \eqref{eq:concavity} for $g=0$.
In GW theory, the invariants are the intersection numbers
\begin{equation}\label{eq:GWinv}
\langle \tau_{a_1}(\fie_{\iii_1}),\dots,\tau_{a_n}(\fie_{\iii_n})
\rangle^{\GW}_{g,n,\delta}= \prod_{i=1}^n \ev^*_i
(\fie_{\iii_i})\psi_i^{a_j} \cap [X_W]_{g,n,\delta}^{\vir}. \end{equation} The
generating functions are respectively
\begin{equation}\label{eq:FJRWpot}
\Fcal_{\RW}^g=
\sum_{\substack{a_1,\dots, a_n\\ \iii_1, \dots, \iii_n}}
\langle \tau_{a_1}(\phi_{\iii_1}),\dots,\tau_{a_n}(\phi_{\iii_n})
\rangle^{\RW}_{g,n}\frac{t_{a_1}^{\iii_1}\cdots t_{a_n}^{\iii_n}}{n!}
\end{equation}and
\begin{equation}\label{eq:GWpot}
\Fcal_{\GW}^g=\sum_{\substack{a_1,\dots, a_n\\ \iii_1, \dots, \iii_n}}\sum _{\delta\ge 0}
\langle \tau_{a_1}(\fie_{\iii_1}),\dots,\tau_{a_n}(\fie_{\iii_n})
\rangle^{\GW}_{g,n,\delta}\frac{t_{a_1}^{\iii_1}\cdots
t_{a_n}^{\iii_n}}{n!}Q^\delta.\end{equation}

In this way, both theories yield a power series
$$\Fcal_{\W}^g=\sum_{\substack{a_1,\dots, a_n\\ \iii_1, \dots, \iii_n}}\sum _{\delta\ge 0}
\langle \tau_{a_1}(\Phi_{\iii_1}),\dots,\tau_{a_n}(\Phi_{\iii_n})
\rangle^{\W}_{g,n,\delta}\frac{t_{a_1}^{\iii_1}\cdots t_{a_n}^{\iii_n}}{n!}Q^\delta$$
in the variables $t_{a}^i$ (for FJRW theory the contribution of
the terms $\delta >0$ is set to zero,
whereas $\langle\ \ \rangle^{\W}_{g,n,0}$
should be read as $\langle\ \ \rangle_{g,n}^{\RW}$).

We can also define the partition function
\begin{equation}\label{eq:totpot}
\Dcal_{\W}=\exp\left(\textsum_{g\ge 0} \hbar^{g-1} \Fcal^g_{\W}\right).\end{equation}

\begin{rem}\label{rem:divisorequation}
By the divisor axiom, $Q $ and $t_0^1$ are related. In fact, they
    appear together in the form $Qe^{t_0^1}$. Now, we assume that
    $\Fcal_{\GW}$ converges when $|Qe^{t_0^1}|$ is sufficiently small
    and set $Q=1$ and $q=e^{t_0^1}$. Therefore, the variable $Q$ will
    be omitted starting from the next paragraph. This can be achieved in
    the region where we have ${\Re}(t_0^1)\ll 0$ (a similar
    and more detailed discussion
    can be found in \cite[``The divisor equation'', p.6]{Coates}).
\end{rem}

\paragraph{The Lagrangian cones.}
Let us focus on the genus-zero potential $\Fcal_{\W}^0$.
The dilaton shift
$$q^{\iii}_a=\begin{cases}
t^0_1-1&\text{if $(a,{\iii})=(1,0)$}
\\t^{\iii}_a &\text{otherwise.}\end{cases}$$
makes $\Fcal^0_\W$ into a power series in the Darboux coordinates
$q^{\iii}_a$. Now we can define $\Lcal_\W$ as
$$\Lcal_\W:=\{\pmb p=d_{\pmb q}\Fcal^0_\W\}\subset\Vcal_\W.$$
With respect to the symplectic form $\Omega_\W$,
the subvariety $\Lcal_\W$ is
a Lagrangian cone
whose tangent spaces
satisfy the geometric condition $zT=\Lcal_\W\cap T$
at any point
(this happens because both potentials
satisfy the equations SE, DE and TRR of \cite{Givental};
in FJRW theory, this is guaranteed by \cite[Thm.~4.2.8]{FJR1}
and will be also observed later as a consequence of Proposition \ref{pro:CZ}).

Every point of $\Lcal_{\W}$
can be written as follows
$$-z\Phi_0+ \sum _{\substack{0\le \iii\le 3\\a\ge 0}}t^{\iii}_a \Phi_\iii z^a+
\sum_{\substack{n\ge 0\\ \delta\ge 0 }}
\ \sum_{\substack{0\le \iii_1,\dots, \iii_n\le 3 \\ a_1,\dots,a_n\ge 0}}
\ \sum_{\substack{0\le \epsilon\le 3\\ k\ge 0 }}\frac{t_{a_1}^{\iii_1}\cdots t_{a_n}^{\iii_n}}{n!(-z)^{k+1}}
\langle \tau_{a_1}(\Phi_{\iii_1}),\dots,\tau_{a_n}(\Phi_{\iii_n}),\tau_k(\Phi_\epsilon)
\rangle^{\W}_{0,n+1,\delta}\Phi^\epsilon,$$
where the term $-z\Phi_0$ performs the dilaton shift.
Setting $a$ and $a_i$ to zero, we obtain the points of the form
\begin{equation}\label{eq:Jlocus}
-z\Phi_0+ \sum _{0\le \iii\le 3}t_0^\iii \Phi_\iii +
\sum_{\substack{n\ge 0\\\delta\ge 0 }} \ \sum_{0\le \iii_1,\dots,
\iii_n\le 3} \ \sum_{\substack{0\le \epsilon\le 3\\ k\ge 0
}}\frac{t_{0}^{\iii_1}\cdots t_{0}^{\iii_n}}{n!(-z)^{k+1}} \langle
\tau_{0}(\Phi_{\iii_1}),\dots,\tau_{0}(\Phi_{\iii_n}),\tau_k(\Phi_\epsilon)
\rangle^{\W}_{0,n+1,\delta}\Phi^\epsilon,
\end{equation}
which define a locus which uniquely determines
the rest of the points of $\Lcal_{\W}$ (via
multiplication by $\exp(\al /z)$ for any $\al \in \CC$---\emph{i.e.} via
the string equation---and via the divisor equation in GW theory).
We define the $J$-function
$$t=\sum_{\iii=0}^3 t_0^\iii \Phi_\iii
\mapsto J_{\W}(t,z)$$ from the state space
$H_{\W}$
to the symplectic vector space $\Vcal_{\W}$ so that
$J_{\W}(t,-z)$ equals the expression \eqref{eq:Jlocus}.
In this way for $t_0^0,t_0^1,t_0^2,t_0^3$ varying in $\CC$
the family $J_{\W}(t=\sum_{\iii=0}^3 t_0^\iii \Phi_\iii,-z)$ varies in $\Lcal_{\W}$.
In fact, this is the only family of elements of $\Lcal_{\W}$ of the
form
$-z+\sum_{\iii=0}^3 t_0^\iii +O(z^{-1}).$

\paragraph{Reconstructing Givental's Lagrangian cone from the small J-function.}
It is well known that the genus-zero GW-theory can be
reconstructed from the {\em big} J-function $t\mapsto
J_{\GW}(t,z)$ via the tautological equations SE, DE and TRR. Here,
{\em big} means that $t$ takes value in the entire state space.
More precisely, Givental's Lagrangian cone can be reconstructed
from knowing that it contains the image  of $t\mapsto
J_{\GW}(t,-z)$ for $t\in H_{\GW}$ (the big $J$-function)---a special family of elements on the cone $\Lcal_{\GW}$. The fact that
$X_W$ is a Calabi--Yau quintic three-fold allows further
reduction: Givental's Langrangian cone can be reconstructed from
the so called {\em small} $J$-function, where {\em small} means
that $t$ takes values only in $H^2(X_W;\CC)=[H]\CC\subset
H_{\GW}$. We detail the arguments on both sides in the following remarks.

    \begin{rem}\label{rem:CYsimpl} Recall that the virtual cycle is $n$-dimensional: therefore
    $\GW$ invariants vanish unless we have $$2n=2\textsum_{i=1}^n a_i+\textsum_{i=1}^n \deg(\fie_{\iii_i}).$$
    (notice that this is precisely the same constraint as \eqref{eq:dimensioncompatibility}
    on the $\RW$ side).
    This condition, together with the string equation eliminating
    $\tau_0(H^0)$ insertions,
    allows us to reduce the computation to the invariants whose entries are
    of the form $\tau_1(H^0)$ or $\tau_0(H^1)$. In this way the
    generating function $\Fcal_{\GW}$ is determined by its restriction to
    the parameters $t_1^0$ and $t_0^1$.
    We can further use the dilaton equation and the divisor equation
    to eliminate $\tau_1(H^0), \tau_0(H^1)$ as well and to
    express everything in terms of the invariant $\langle\ \ \rangle_{0,n,\delta}$.
    More specifically, as mentioned in Remark
    \ref{rem:divisorequation},
    the divisor axiom tells us that
    $t_0^1$ appears in $\Fcal_{\GW}$ in the form $e^{t_0^1}$;
    so, the entire $\GW$ potential is determined
    by a one parameter function in the parameter $q=e^{t_0^1}$.
    \end{rem}

    \begin{rem}\label{rem:FJRWsimpl}
    The same holds for FJRW theory. As we already argued
    in Remark \ref{rem:eveninsertions} we can express
    $\Fcal_{\RW}$ in terms of its restriction on the parameters $t_1^0$ and $t_0^1$.
    We can use the dilaton equation to eliminate $t_1^0$.
    We are left with a single parameter $t_0^1$; the
    difference here is that we can not eliminate $t_0^1$ since
    there is no divisor equation (this is compensated by the absence
    in the theory of $\RW$ invariants $\langle\dots\rangle_{g,n}^{\RW}$ of the parameter
    specifying the degree of the stable map appearing in the theory of $\GW$ invariants
    $\langle\ \dots \rangle_{g,n,\delta}^{\GW}$).
    \end{rem}

\subsection{The conjecture}
The following conjecture can be regarded as a geometric version of
the physical Landau--Ginzburg/Calabi--Yau correspondence
\cite{VW} \cite{Wi3}. A precise
    mathematical conjecture is proposed by the second author
    \cite{R} and applies to a much more general category
    including Calabi--Yau complete intersections inside weighted projective
    spaces. In order to keep the notation simple, we review the conjecture for
    the case of quintic three-folds. The formalism is
    analogous to the
    conjecture of \cite{CR} and \cite{CIT} on crepant resolutions of
    orbifolds and uses Givental's quantization from \cite{Givental}, which is naturally defined
    in the above symplectic spaces $\Vcal_{\RW}$ and $\Vcal_{\GW}$.

\begin{conj}\label{conj}
    Consider the Lagrangian cones $\Lcal_{\RW}$ and $\Lcal_{\GW}$.
    \begin{enumerate}
    \item[(1)] There is a degree-preserving $\CC[z, z^{-1}]$-valued linear symplectic isomorphism
       $$\mathbb U: \Vcal_{\RW}\rightarrow \Vcal_{\GW}$$ and a choice of analytic
       continuation of $\Lcal_{\RW}$ and $\Lcal_{\GW}$
    such that $\mathbb U (\Lcal_{\RW})=\Lcal_{\GW}.$
    \item[(2)] Up to an overall constant, the total potential functions up to a choice of
    analytic continuation are related by
    quantization of $\mathbb U $; {i.e.}
    $$\Dcal_{\GW}=\widehat{\mathbb U}(\Dcal_{\RW}).$$
    \end{enumerate}
\end{conj}

    \begin{rem}
    For the readers familiar with the crepant resolution
    conjecture \cite{CR} and \cite{CIT}, an important difference here is the lack
    of monodromy condition.\end{rem}

    By \cite{CR}, a direct consequence of the first part of the
    above conjecture is the following
    isomorphism between quantum rings.

\begin{cor}\label{cor:quantumiso}
For an explicit specialization of the variable $q$ determined by
$\mathbb U$, the quantum ring of $X_W$ is isomorphic to the quantum ring
of the singularity $\{W=0\}$.
\end{cor}

    We refer readers to \cite{CR} for the derivation of the above isomorphism.

\section{The correspondence for the quintic three-fold in genus zero}\label{sect:comp}

\subsection{Determining the Lagrangian cone of FJRW theory}
\paragraph{The twisted Lagrangian cone.}\label{sect:twistedcone}
This section introduces a symplectic space
$\Vcal_{\tw}$ and a  Lagrangian cone
$\Lcal_{\tw}$, which approximate the symplectic space
$\Vcal_{\RW}$ and the Lagrangian cone $\Lcal_{\RW}$.
This will ultimately allow us to determine the function
$J_{\RW}$ and the entire Lagrangian cone $\Lcal_{\RW}$.
A similar story, with a different symplectic space $\Vcal_{\tw}$,
holds on the Gromov--Witten side,
see \cite{Gi}, \cite{Gib} and the systematic
treatment of twisted potentials in \cite{CG}.
The definition of $\Vcal_{\tw}$ is based on two observations: Lemma \ref{lem:assemble}
and Lemma \ref{lem:limit}.

First, the generating function
$\Fcal_{\RW}^0$ defined on space $\{t_a^h\in \CC \mid a\ge 0, 0\le h\le 3\}$
(\emph{i.e.} on the space $\Vcal_{\RW}^+$),
can be extended the space $\{t_a^h\in \CC \mid a\ge 0, 0\le h\le 4\}$: we simply set
$\langle \tau_{a_1}(\phi_{\iii_1}),\dots,\tau_{a_n}(\phi_{\iii_n})
\rangle^{\RW}_{0,n}=0$ as soon as $\iii_j=4$ for some $j$.
Equation \eqref{eq:concavity} expressing the $\RW$ invariants in terms
of intersections of cycles in $(A_4)_{0,n}$
generalizes to this setting; we need to slightly modify the moduli functor.
For $m_1,\dots, m_n\in \{1,\dots, 5\}$, consider the stack
 $\wt A_4\textstyle{\left(\frac{m_1}5,\dots,\frac{m_n}5\right)}_{g,n}$
classifying
genus-$g$ $n$-pointed
$5$-stable curves equipped with $5$th roots
\begin{equation}
\label{eq:newmoduli}
\wt A_4\textstyle{\left(\frac{m_1}5,\dots,\frac{m_n}5\right)}_{g,n}:=
\bigg\{(C;(\sigma_i)_{i=1}^n; T; \fie)\mid \fie\colon
T^{\otimes 5}\xrightarrow{\ \sim \ } \omega_{\log}(-D_{\pmb m})
 \text{ and  $\mult_{\sigma_i}(T)=0 \ \forall i$}\bigg\},
 \end{equation}
where $D_{\pmb m}$ is the linear combination $m_1D_1+\dots+m_nD_n$ of
the integer divisors $D_i$ corresponding to the markings $\sigma_i$.
\begin{lem}\label{lem:assemble}
Let $\iii_1,\dots,\iii_n\in \{0,\dots, 4\}$. Let $\pi$ be the universal family and
$\Tcal$ be the universal root of the moduli functor
$\wt A_4\textstyle{\left(\frac{\iii_1+1}5,\dots,\frac{\iii_n+1}5\right)}_{0,n}$.
Then, $R^1\pi_* \Tcal$ is locally free and $\pi_*\Tcal$ vanishes.
Furthermore, we have
\begin{equation}\label{eq:getthere2}
\langle \tau_{a_1}(\phi_{\iii_1}),\dots,\tau_{a_n}(\phi_{\iii_n})
     \rangle^{\RW}_{0,n}=
{5} \prod_{i=1}^n \psi_i^{a_i}
c_{\rm top}((R^1\pi_*\Tcal)^\vee)^5 \cap
\left[\wt A_4\textstyle{\left(\frac{\iii_1+1}5,\dots,\frac{\iii_n+1}5\right)}\right]_{0,n}.
\end{equation}
\end{lem}
\begin{proof}
Let us compare the definition of $\wt A_4$ and that
of $A_4$ first. For $m_1,\dots,m_n$
ranging in $[0,4]$ the moduli stack
$\wt A_4\textstyle{\left(\frac{m_1}5,\dots,\frac{m_n}5\right)}_{0,n}$
are canonically isomorphic to
the moduli stacks $A_4(J^{m_1},\dots,
    J^{m_{n}})_{0,n}$ (an easy consequence of
Lemma \ref{lem:exists} and the fact that there is a natural equivalence
between $5$th roots of two line bundles whenever they differ by $\Ocal(-5D_i)$ for some $i$).
It is crucial however, to observe that the universal objects differ:
under the identification between the moduli stacks we may relate the two universal
$5$th roots $\Tcal$ and $\Lcal$ on the same orbifold curve $(\Ccal;\sigma_1,\dots,\sigma_n)$ as
follows
\begin{equation}\Tcal=p_*\Lcal\otimes \Ocal(-\textsum_{i\mid m_i=5} D_i),
\end{equation}
where $p$ is the morphism forgetting the stabilizer at each marking.

With this comparison understood, we already know that the claim is true
if all entries
$\iii_1, \dots, \iii_n$ differ from $4$. In order to prove the rest of the claim it
is enough to show that $c_{\rm top}((R^1\pi_*\Tcal)^\vee)$
vanishes as soon as ${\iii_{i_0}}=4$ for
some $i_0\in \{1,\dots, n\}$.
Consider the integer divisor  corresponding to the
$i_0$th marking $D_{i_0}$ and the exact sequence
$$0\to \Tcal\to \Tcal(D_{i_0})\to \Tcal(D_{i_0})\rest{D_{i_0}}\to 0.$$
We may regard the restriction $\Tcal(D_{i_0})\rest{D_{i_0}}\cong
\Lcal\rest{D_{i_0}}$ as a $5$th root of
$\omega_{\log}\rest{D_{i_0}}\cong \Ocal_{D_{i_0}}$; \emph{i.e.} a
line bundle with trivial first Chern class in rational cohomology
\begin{equation}\label{eq:c1zero}
c_1(\pi_*(\Tcal(D_{i_0})\rest{D_{i_0}}))=0.
\end{equation}

Notice that the restriction of $\Tcal$ to every fibre of the
universal curve has no global sections. Indeed,
the same argument of Remark \ref{rem:concavity} applies: on a
smooth fibre the degree is negative, whereas on a reducible fibre
$C$ we observe that the degree of $\Tcal$ on an irreducible
component $Z$ is less than $\#(Z\cap \ol{C\setminus Z})-1.$ By a
simple induction argument on the number of components, this
implies that there are no nonzero global sections on the
fibres of $\pi$; hence, $R^1\pi_* \Tcal$ is a vector bundle and
 $\pi_*\Tcal$ vanishes.

We further notice that the same holds for $\Tcal(D_{i_0})$.
In fact, for $\Tcal$, the above induction applies
by iteratively removing tails
(rational irreducible components attached to the rest of the
fibre at a single point).
For this argument it is crucial to observe that
$d_Z<\#(Z\cap \ol{C\setminus Z})-1$ simply reads
$d_Z<0$ on a tail. On the other hand, for $\Tcal(D_{i_0})$, it may well happen that the
component carrying the $i_0$th point does not satisfy
$d_Z<\#(Z\cap \ol{C\setminus Z})-1$ because of the twisting at $D_{i_0}$; however, we can still apply
the induction argument by starting from a different tail
(reducible curves of genus zero have at least two tails).

By the vanishing of $\pi_* \Tcal$ and $\pi_*\Tcal(D_{i_0})$,
we now have the following exact sequence of vector bundles
$$0\to \pi_*\Tcal(D_{i_0})\rest{D_{i_0}}\to R^1\pi_*\Tcal \to R^1\pi_*\Tcal(D_{i_0})\to 0,$$
which immediately yields $\ctop(R^1\pi_*\Tcal)=0$ via \eqref{eq:c1zero} as required.
\end{proof}

The second observation allows us to set a relation between the well known
intersections $\prod_i \psi_i^{a_i}$
and the numbers \eqref{eq:getthere2}.
Since the two numbers differ because of the class
$c_{\rm top}((R^1\pi_*\Tcal)^\vee)^5$, we now define  a
class interpolating the fundamental class
$[\wt A_4\textstyle{(\frac{\iii_1+1}5,\dots,\frac{\iii_n+1}5)}]_{0,n}$
and
$c_{\rm top}((R^1\pi_*\Tcal)^\vee)^5\cap [\wt A_4\textstyle{(\frac{\iii_1+1}5,\dots,\frac{\iii_n+1}5)}]_{0,n}$.
\begin{lem}\label{lem:limit}
For any choice of
complex parameters $s_d$ with $d\ge 0$ define on the $K$ theory ring
the cohomology class
\begin{align*}
K_0(X)&\to H^*(X;\CC)\\
x&\mapsto \exp\left(\textsum_{d}s_d\ch_d(x)\right).
\end{align*}
The above class is multiplicative for any $s_d$ with $d\ge 0$.
For $s_d=0 \ \forall d$ we get
the fundamental class, whereas for
\begin{equation}\label{eq:fixs}
s_d=\begin{cases}
-5\ln(\lambda) & d=0,\\\\
\displaystyle{\frac{5(d-1)!}{\la^d}} & d>0.
\end{cases}
\end{equation}
the multiplicative class is related to the equivariant Euler class as follows
$$
\exp\left(\textsum_{d}s_d\ch_d(-[V])\right)=e_{\CC^\times}(V^\vee)^5,$$
where $V$ is a $\CC^\times$-equivariant vector bundle $V$
equipped with the natural
action of $\la\in \CC^\times$ scaling the fibres by multiplication.
In particular the nonequivariant limit
$\lim _{\la\to 0}$ yields $c_{\rm top}(V^\vee)^5$.
\end{lem}
\begin{proof}
For any vector bundle $V$
equipped with the $\CC^\times$-action scaling the fibre; the
$\CC^\times$-equivariant Euler class can be expressed in terms of the
nonequivariant Chern character as follows\footnote{The above relation can be easily checked on a line bundle and
extends to any vector bundle by the splitting principle;
for any line bundle $L$
with $\CC^\times$ acting by multiplication along
the fibres, we can express
the $\CC^\times$-equivariant Euler class in terms of the nonequivariant
first Chern class $\la + c_1(L)$ and ultimately as
\begin{multline*}
\la +c_1(L)=\la\left(1+\frac{c_1(L)}{\la}\right)=
\la\exp\left(\ln\left(1+\frac{c_1(L)}{\la}\right)\right)
\\=\la \exp\left(\textsum_{d>0}   (-1)^{d-1}\frac{c_1(L)^d}{\la^d d}\right)
=
\exp\left(\ln(\la)\ch_0(L)+\textsum_{d>0} (-1)^{d-1}\frac{(d-1)!}{\la^d} \ch_d(L)\right).\end{multline*}}
:
$$e_{\CC^\times}(V)=
\exp\left(\ln(\la)\ch_0(V)+\textsum_{d>0} (-1)^{d-1}\frac{(d-1)!}{\la^d} \ch_d(V)\right).$$

Finally, we can show that, with the above parameters $s_d$, we have
\begin{multline*}
e_{\CC^\times}(V^\vee)^5=
\exp\left(5\ln(\la)\ch_0(V^\vee)+
\textsum_{d>0} 5(-1)^{d-1}\frac{(d-1)!}{\la^d} \ch_d(V^\vee)\right)\\=
\exp\left(-5\ln(\la)\ch_0(-[V]^\vee)+\textsum_{d>0}
5\frac{(d-1)!}{\la^d} \ch_d(-[V])\right)
=\exp\left(\textsum_{d}s_d\ch_d(-[V])\right),\end{multline*}
where the relations
$\ch_d(-V)=-\ch_d(V)$ and $\ch_d(V^\vee)=(-1)^d\ch(V)$ have been employed.
\end{proof}

The two previous observations (Lemmata \ref{lem:assemble} and \ref{lem:limit})
justify the
introduction of a fifth generator for the state space,
the definition of the intersection numbers over the moduli stacks
of $\wt A_4$-curves, and the
definition of the  state space over an extended ground ring
$\CC[\la]\otimes \CC[[s_0,s_1,\dots]].$ We start from the twisted symplectic space.

The twisted symplectic vector space
$\Vcal_{\tw}$
 is formed by Laurent series over a modified
 state space. We have
$$\Vcal_{\tw}=H_{\tw}\otimes \CC((z^{-1})),$$
where the state space is modified by adding a new element
$\phi_4$ to the base
$$H_{\tw}=\phi_0 R\oplus \phi_1 R\oplus \phi_2 R\oplus \phi_3 R \oplus \phi_4 R$$
and by working on the extended ground ring
$$R=H^*_{\CC^\times}(\pt;\CC)[[s_0,s_1,\dots ]]$$
of power series in infinitely many variables $s_0,s_1,\dots$ with values in the
ring of $\CC^\times$-equivariant cohomology of a point.
We regard $H^*_{\CC^\times}(\pt;\CC)$ as $\CC[\lambda]$; so
we can identify $R$ with
$\CC[\la]\otimes \CC[[s_0,s_1,\dots]]$.
The nondegenerate inner product $( \ ,\ )_{\RW}$
is extended to $H_{\tw}$ by setting
$$(\phi_4,\phi_\iii)_{\tw} = \begin{cases}\exp(-s_0) &\text{for $\iii=4$}\\
0 &\text{otherwise}
\end{cases}$$
(this choice allows us to use the invariants defined below at
\eqref{eq:twinv} to form a Lagrangian cone $\Lcal_{\tw}$, see in
particular  Proposition \ref{pro:CZ} and the role played by this
pairing in the second part of the proof). In this way we get a
dual basis $\phi^0,\dots, \phi^4$ as soon as we set
\begin{equation}\label{eq:phibase}
\phi^4=\exp(s_0)\phi_4.\end{equation}
We extend Convention \ref{conv:mod5} to these extended bases:
$\phi_{h+5k}=\phi_{h}$ for all $k$.

The symplectic form making $\Vcal_{\tw}$ into a symplectic vector
space is $$\Omega_{\tw} (f_1,f_2)=\Res_{z=0}(f_1(-z),f_2(z))_{\tw}.$$
Again, $\Vcal_{\tw}$ equals
the total cotangent space
$T^*\Vcal^+_{\tw}$, where $\Vcal_{\tw}^+$ is
defined as
$H_{\tw}\otimes \CC[z]$. The Darboux coordinates are
again $\{q^{\iii}_a, p_{l,j}\}$, but this time $\iii$ and $j$ vary in $[0,4]$.
In view of an approximation of $\RW$ invariants
\begin{equation}\label{eq:twinv}
\langle \tau_{a_1}(\phi_{\iii_1}),\dots,\tau_{a_n}(\phi_{\iii_n})
\rangle^{\tw}_{0,n}:= 5 \int_{\wt A_4\left(\frac{\iii_1+1}5,\dots,\frac{\iii_n+1}5\right)_{0,n}}
\prod_{i=1}^n \psi_i^{a_i} \cup \exp\left(\sum_{d\ge 0} s_d \ch_d(R\pi_* \Tcal)\right),
\end{equation}
where we point out that the integral is taken on the standard
fundamental cycle of
$\wt A_4\left(\frac{\iii_1+1}5,\dots,\frac{\iii_n+1}5\right)_{0,n}$.
The definitions of $\Fcal^0_{\W}$, $\mathcal D_{\W}$, and
$\Lcal_{\W}$ of \S\ref{sect:lagcone} generalize word for word and yield
the twisted potential $\Fcal^0_{\tw}$, the partition function $\mathcal D_{\tw}$ and
the locus $$\Lcal_{\tw}\subset \Vcal_{\tw}.$$
In Proposition \ref{pro:CZ}, we
prove that this locus is a Lagrangian cone; let us first remark two special
cases of the construction.
\begin{rem}[the untwisted cone]
If we set $s_d=0$ for all $d\ge 0$, then we get
the untwisted symplectic vector space
$\Vcal_{\un}$ of Laurent series over the state space $H_{\un}:=\bigoplus _{\iii=0}^4 \phi_\iii R$
with $\phi^4=\phi_4$ by \eqref{eq:phibase}.
Correspondingly, we get the
cone $\Lcal_{\un}$
encoding the intersection numbers
$$\langle \tau_{a_1}(\phi_{\iii_1}),\dots,\tau_{a_n}(\phi_{\iii_n})
\rangle^{\un}_{0,n}= 5\int_{\wt A_4\left(\frac{\iii_1+1}5,\dots,\frac{\iii_n+1}5\right)_{0,n}}
\prod_{i=1}^n \psi_i^{a_i}.$$
It is easy to prove by carrying out the calculation directly on $\MMM_{0,n}$
that these numbers
equal
\begin{equation}\label{eq:untwistinv}
\begin{pmatrix}{\sum_i a_i}\\{a_1,\dots,a_n}\end{pmatrix}
\end{equation}
as soon as $n-3=\sum_i a_i$ and
the following selection rule guaranteeing
$\wt A_4\left(\frac{\iii_1+1}5,\dots,\frac{\iii_n+1}5\right)_{0,n}\neq \varnothing$
is satisfied
\begin{equation}\label{eq:selrule}
\qquad 2+\textsum_i \iii_i\in 5\ZZ
\end{equation}
(the factor $5$ in the definition of $\langle\ \ \rangle^{\tw}$ is canceled by
the degree computation $[\wt A_4:\MMM_{0,n}]=1/5$).
The untwisted cone is a Lagrangian cone
whose tangent spaces
satisfy the geometric condition $zT=\Lcal_W\cap T$
at any point (this happens because the
equations SE, DE and TRR of \cite{Givental}
can be easily deduced from the analogue equations on $\MMM_{0,n}$).
\end{rem}
\begin{rem}[the Lagrangian cone $\Lcal_{\RW}$]\label{rem:fixs}
Let us set $s_d$ as in \eqref{eq:fixs}.
In this way we have $\phi_4=\la^{5}\phi^4$.
By Lemma \ref{lem:assemble} the higher direct image
$-R\pi_*\Tcal$ is represented by
the \emph{locally free} sheaf $R^1\pi_*\Tcal$.
By Lemma \ref{lem:limit} we have
$$\exp\left(\textsum_{d}s_d\ch_d(R\pi_*\Tcal)\right)=
\exp\left(\textsum_{d}s_d\ch_d(-R^1\pi_*\Tcal)\right)=e_{\CC^\times}((R^1\pi_*\Tcal)^\vee)^5.$$
Therefore the nonequivariant limit $\lim_{\la\to 0}$ of
the twisted invariants $\langle \tau_{a_1}(\phi_{\iii_1}),\dots,\tau_{a_n}(\phi_{\iii_n})
\rangle^{\tw}_{0,n}$ is the intersection number
$$5 \int_{\wt A_4\left(\frac{\iii_1+1}5,\dots,\frac{\iii_n+1}5\right)_{0,n}}
\prod_{i=1}^n \psi_i^{a_i} \cup (c_{\rm top}(R^1\pi_* \Tcal))^5.$$
As a consequence of Lemma \ref{lem:assemble} we have
\begin{equation}\label{eq:limit}
\lim_{\la\to 0}\left(\Fcal^0_{\tw}\rest{s_d=\ol s_d}\right)=\Fcal^0_{\RW}\end{equation}
when  $\ol s_d$ are the parameters $s_d$ fixed as  in \eqref{eq:fixs}.
Indeed, if
all entries $\iii_i$ are different from $4$ the identification with
the numbers
$\langle \tau_{a_1}(\phi_{\iii_1}),\dots,\tau_{a_n}(\phi_{\iii_n})
\rangle^{\RW}_{0,n}$ is immediate, whereas
if $\iii_i=4$ for some $i$ the limit of the
twisted invariant
$\langle \tau_{a_1}(\phi_{\iii_1}),\dots,\tau_{a_n}(\phi_{\iii_n})
\rangle^{\tw}_{0,n}$ vanishes.

In this way we conclude that the Lagrangian
cone $\Lcal_{\RW}$ can be realized as the nonequivariant limit (for $\la\to 0$)
of $\Lcal_{\tw}$ in the $\{\phi^i\}$-frame
with the parameters $s_d$ set to the above values.
\end{rem}

The following
proposition shows that the entire Lagrangian cone $\Lcal_{\tw}$ can be
reconstructed from $\Lcal_{\un}$.
\begin{pro}\label{pro:CZ}
Consider the symplectic transformation
$\Delta\colon \Vcal_\un \to \Vcal_\tw$
given by the direct sum
\begin{equation}\label{eq:transf}
\Delta=\bigoplus_{i=0}^{4}
\exp \left(\sum_{d\ge 0} s_d \frac{B_{d+1}\left(\frac{i+1}{5}\right)}
{(d+1)!}z^d
\right).\end{equation}
We have \begin{equation}\label{eq:Lagr}
\Lcal_\tw=\Delta (\Lcal_\un).
\end{equation}
\end{pro}
\begin{proof}
The claim is proven in \cite{CZ}, but in a slightly different setting:
in \cite{CZ} both the twisted and the untwisted cone correspond to
intersection numbers over Witten's top Chern class
$c^{\vir}_\W\left({\pmb{\iii}}\right)$.
With a diagram, we can summarize the main theorem of \cite{CZ} as
\begin{align*}
&&\prod_i\psi_i^{a_i}\cap \textstyle{c^{\vir}_\W\left({\pmb{\iii}}\right)}
&& \overset{\un/\tw}\longrightarrow
&&\prod_i\psi_i^{a_i} \exp\left(\textsum_d s_d \ch_d(R\pi_*\Tcal)\right)\cap
\textstyle{c^{\vir}_\W\left({\pmb{\iii}}\right)},
\end{align*}
where the numbers appearing on the left hand side are those encoded in
the untwisted Lagrangian cone  of \cite{CZ}, whereas the number on the right hand side
are those encoded in the twisted Lagrangian cone of \cite{CZ}.
In genus zero, the cycle $c_\W\left({\pmb{\iii}}\right)$ is simply
the Poincar\'e dual of the top Chern class $\ctop((R^1\pi_*\Tcal)^\vee)$
on the moduli stack $\wt A_4(\frac{h_1+1}5,\dots,\frac{h_n+1}5)_{0,n}$. As we illustrated
in Lemma \ref{lem:assemble} this class vanishes as soon as one of the entries of
$\pmb{\iii}$ equals $4$.

The present statement claims
that the operator $\Delta$ sets the
relation between following untwisted and twisted invariants
\begin{align*}
&& \prod_i\psi_i^{a_i}
&&\overset{\un/\tw}\longrightarrow
&&\prod_i\psi_i^{a_i} \exp\left(\textsum_d s_d \ch_d(R\pi_*\Tcal)\right).
\end{align*}
Exactly as in \cite{CZ} the quantization of the operator $\Delta$ yields
a differential operator relating the higher genus partition functions $\Dcal_{\un}$ and $\Dcal_{\tw}$:
$\widehat{\Delta}(\Dcal_{\un})=\Dcal_{\tw}$.
This implies the desired claim $\Delta(\Lcal_{\un})=\Lcal_{\tw}$; indeed,
after passing to the quasi-classical limit $\hbar\to 0$,
applying the operator $\Delta$ to $\Dcal_{\un}$ corresponds
to transforming the Lagrangian submanifold $\Lcal_{\un}$
by the (unquantized) operator $\Delta$.

In one respect the passage from these untwisted to these twisted invariants
is simpler compared to \cite{CZ}: it involves the factorization
properties of the fundamental cycle, which are much easier to
prove than those of Witten's top Chern class. We state the crucial
factorization property for the fundamental cycle. Fix a subset
$I\subseteq[n]:=\{1,\dots,n\}$ and $l$ between $0$ and $g$.
Consider the following substack $\Scal_{l,I}$ of the stack $\Scal$ introduced in Section \ref{sect:enum}:
the objects are triples of the form ($\wt A_4$-curve, node, a choice of a branch
of the node) where
the node divides the curve into two
subcurves and the curve corresponding to the chosen branch has
genus $l$ and  marking set $I$.
Let $q\in\{0,\dots, 4\}$ be
equivalent to $2l-2-\textsum_{i\in I} \iii_i \mod 5$.
Write $q'$ for the index in $\{0,\dots, 4\}$
satisfying $q+q'=3 \mod 5$; write $I'$ for $[n]\setminus I$; let $\pmb \iii_H$ be the
multiindex $(\iii_i\mid i\in H)$ for any set $H\subseteq [n]$.
There exist natural morphisms
\[
\wt A_4\left(\tfrac{\pmb \iii_I+\pmb 1}5,\tfrac{q+1}5\right)_{l,\#{I}+1}
\times \wt A_4\left(\tfrac{\pmb \iii_{I'}+\pmb 1}5,\tfrac{q'+1}5\right)_{g-l,\#{I'}+1}
\xleftarrow{\ \ \mu_{l,I}\ }
\Scal_{l,I}\xrightarrow{\ j_{l,I}\ }
\wt A_4\left(\tfrac{\pmb \iii+\pmb 1}5\right)_{g,n},
\]
where $j_{l,I}$ is simply the morphism
forgetting the node
and $\mu_I$ is obtained by normalizing the curve at the node and pulling back.
Here, we should notice that the case when $q=q'=4$ is special: after pulling
back we should apply to
the line bundles obtained on each component
the functor $\mathcal E \mapsto \mathcal E \otimes \Ocal(-D)$ where $D$ is the
integer divisor associated to the point
lifting the node; in this way we obtain an object of
the product
$$\wt A_4\left(\tfrac{\pmb \iii_I+\pmb 1}5,\tfrac{q+1}5\right)_{l,\#{I}+1}
\times \wt A_4\left(\tfrac{\pmb \iii_{I'}+\pmb 1}5,\tfrac{q'+1}5\right)_{g-l,\#{I'}+1}\qquad \text{ with $q=q'=4$}.$$
We can finally state the required factorization property for the
fundamental classes. We have
\begin{equation*}
\mu_{l,I *} j_{l,I}^*\left[\wt A_4( \tfrac{\pmb\iii+\pmb 1}5)\right]_{g,n}=
\mu_{l,I *} [\Scal_{l,I}]=
\left[\wt A_4\left(\tfrac{\pmb \iii_I+\pmb 1}5,\tfrac{q+1}5\right)\right]_{l,\#{I}+1,5}\times
\left [\wt A_4\left(\tfrac{\pmb\iii_{I'}+\pmb 1}5,\tfrac{q'+1}5\right)\right]_{g-l,\#{I'}+1,5}.
\end{equation*}
The crucial fact is that the degree of  $\mu_I$ equals one\footnote
{The stabilizers on the
generic points of both sides are isomorphic to
$\pmmu_5\times \pmmu_5$: on the right-hand side the
two generators act by rotating the fibres of
each $\wt A_4$-structure, whereas on
the left-hand side one generator rotates the fibres of
the $\wt A_4$-structure and the other operates on the curve by means
of an automorphism acting locally as $(x,y)\mapsto (\xi_5 x,y)$ and commonly named the ``ghost'' automorphism.
See discussion of ``$\mu_{l,I}$'' in \cite[p.10]{CZ} for a proof.}.
Similarly, and precisely as in \cite{CZ}, we can define the substack $\Scal_{\irr,q}$ of
$\Scal$ whose objects are triples ($\wt A_4$-curve, node, a choice of a branch
of the node) where
the node is nonseparating and the local index at the
chosen branch is $(q+1)/5$. We have
\begin{equation*}
\mu_{\irr,q *} j_{\irr,q}^*\left[\wt A_4( \tfrac{\pmb\iii+\pmb 1}5)\right]_{g,n}=
\mu_{\irr,q *} [\Scal_{\irr,q}]=
\left[\wt A_4\left(\tfrac{\pmb \iii_I+\pmb 1}5,\tfrac{q+1}5,\tfrac{q'+1}5\right)\right]_{g-1,n+2,5}.
\end{equation*}

The above description extends word for word the description of factorization
properties of Witten's top Chern class
given in \cite{CZ}.
There is one point, however, where the presence of Witten's top Chern class simplifies
things in \cite{CZ}:
since the cycle $\textstyle{c_\W\left({\pmb{\iii}}\right)}$
 vanishes as soon
as one of the entries of
$\pmb \iii$ equals $4$,
in \cite{CZ} we did not need to carry out the calculation
outside the range $[0,3]$. We need to
go through the argument focusing on
this situation. This will allow us to illustrate how the claim essentially
follows from
Proposition \ref{pro:GRR}.
Indeed, the quantization of $\oplus_{i=0}^4 {B_{l+1}\left(\frac{i+1}{5}\right)}
z^l/{(l+1)!}$ is the differential operator \begin{multline}\label{eq:diffop}
P_l=
\frac{B_{l+1}(\frac15)}{(l+1)!}
\frac{\partial}{\partial t_{l+1}^{0}}
-\sum_{\substack{a\ge 0\\ 0\le \iii\le 4}}\frac{B_{l+1}(\frac{\iii+1}5)}{(l+1)!}
t_{a}^{\iii}\frac{\partial }{\partial t_{a+l}^{\iii}}
+ \frac{\hbar}{2}\sum_{\substack{a+a'=l-1\\0\le \iii, \iii'\le 4}}(-1)^{a'}
g^{\iii, \iii'}
\frac{B_{l+1}(\frac{\iii+1}5)}{(l+1)!}\frac{\partial ^2}{\partial t_{a}^{\iii} \partial
t_{a'}^{\iii'}};
\end{multline}
notice that
the upper indices $\iii$ and $\iii'$ range in $[0,4]$
and we used the convention $(g^{\iii,\iii'})=(g_{\iii,\iii'})^{-1}$ for $(g_{\iii,\iii'})=(\ ,\ )_{\tw}$.
Therefore, we need to show that
$\partial \Dcal_{\tw} /\partial s_l$ equals
$P_l \Dcal_{\tw}$; which is Proposition \ref{pro:GRR} and can be proven as
in \cite{CZ} with a minor modification when  $g^{\iii,\iii'}=g^{4,4}$. There
the following equations \eqref{eq:sector4} and
\eqref{eq:sector4irr} are involved.

Notice first that for $l=0$ the third summand is not involved and
the above statement simply says that the insertion of $\ch_0(R\pi_*L)$
is equivalent to the multiplication by
$$\left(-\frac12+\frac15\right)(2g-2+n)-\sum_{i=1}^n\left(-\frac12+\frac{h_i+1}5\right)=
-g+1+\frac15(2g-2+n)-\sum_{i=1}^n \frac{h_i+1}{5}.$$ This is indeed the multiplicity
of $\ch_0$ by Riemann--Roch.

For $l>0$, the third summand is involved. Notice, however, that
the coefficient $g^{\iii,\iii'}$
appearing there is different from zero only if
$(\iii,\iii')$ equals $(0,3),(1,2),(2,1),(3,0),$ or $(4,4)$. Therefore,
the only new contribution with respect to \cite{CZ}
occurs for $(\iii,\iii')=(4, 4)$.
The desired formula follows as in \cite[Prop.~4.3.1, Step 3 and 4]{CZ}
using the following
geometric property.
The higher direct image in $K$-theory of a line bundle $L$
on a family of curves $C\to X$ having a node $n\colon X\to C$
differs from the higher direct image
of the pullback of $L$ on the family $C^\nu \to X$
obtained by normalization at the node $n$. The difference is precisely
the $K$-class $-n^*L$: in our situation $L$ is the $5$th root of $\omega$ and
$\ch(-n^*L)=-\pmb 1$. Summarizing, whenever
$I\subseteq [n]$ and $l\in \{0,\dots,g\}$
satisfy $2l-2-\sum_{i\in I} \iii_i\equiv 4 \mod 5$,
we have
$$
\mu_{l,I *}j_{l,I}^*(\ch_d)=\begin{cases}
(\ch_d\times 1)+(1\times \ch_d)&\text{if $d>0$}\\
(\ch_0\times 1)+(1\times \ch_0)+(1\times 1)&\text{if $d=0$},\end{cases}$$
(where we write $\ch_d$ for $\ch_d(R\pi_*\Tcal)$) and, therefore,
\begin{equation}\label{eq:sector4}
\mu_{l,I *}j_{l,I}^*(\exp(\textsum_d s_d\ch_d))=\exp(s_0)
\exp(\textsum_d s_d [\ch_d\times 1])\exp(\textsum_d s_d [1\times \ch_d]).
\end{equation}
The appearance of the factor $\exp(s_0)$ matches the new boundary contribution
of \eqref{eq:diffop}: indeed, $g^{4,4}$ equals $\exp(s_0)$. The analogous
relation holds for
the morphisms $\mu_{\irr ,4}$ and $j_{\irr,4}$:
\begin{equation}\label{eq:sector4irr}
\mu_{\irr,4 *}j_{\irr,4}^*(\exp(\textsum_d s_d\ch_d))=\exp(s_0)
\exp(\textsum_d s_d \ch_d).
\end{equation}
\end{proof}

\paragraph{A family on the FJRW Lagrangian cone.}
In this section, we exhibit a
family of the form
$$I_{\RW}(t, z)\in f(t)z+H_{\RW}[[z^{-1}]] \qquad \text{with $f(t)\in H_{\RW}$}$$
such that $I_{\RW}(t,-z)$ belongs to the Lagrangian cone $\Lcal_{\RW}$.
We will use it to determine the entire Lagrangian cone $\Lcal_{\RW}$.
In the next statement the Pochhammer symbols
$[a]_n=a(a+1)\dots(a+n-1)$ have been used.
\begin{thm}\label{thm:FJRW-I-funct}
Let
$$I_{\RW}(t,z)= z\sum_{k=1,2,3,4} \omega_k^{\RW}(t) \frac{1}{z^{k-1}}\, \phi^{4-k}\qquad \qquad \text{with}\qquad
\omega_k^{\RW}(t)=\frac{1}{\Gamma(k)}\sum_{l\ge 0} \frac{([\frac{k}{5}]_l)^5}{[k]_{5l}}{t^{k+5l}}.$$
The family $\CC\ni t \mapsto I_{\RW}(t,-z)$  lies on the Lagrangian cone $\Lcal_{\RW}$.
\end{thm}
\begin{proof}
We proceed as follows:
\begin{enumerate}
\item
Using the above explicit presentation \eqref{eq:Jlocus} of $J_{\un}$,
we get a function taking values in $\Lcal_{\un}$.
\item
We operate on $J_{\un}$ with a transformation which sends
it to a new function taking values in $\Lcal_{\un}$.
\item
We apply $\Delta$
to this function; we get a function $I_{\tw}(t, z)$ taking values on
$\Lcal_{\tw}$.
\item By setting $s_0,s_1,\dots$ as in Lemma \ref{lem:limit} and
by taking the limit ($\la\to 0$), we get $I_{\RW}$.
\end{enumerate}

Step 1. We recall the definition of the $J$-function for the untwisted cone
\begin{multline*}
J_{\un}(\textsum _{0\le \iii\le 4}t_0^\iii,-z)=
-z\phi_0+ \textsum _{0\le \iii\le 4}t_0^\iii \phi_\iii +\\
\sum_{\substack{0\le \iii_1,\dots, \iii_n\le 4\\{k\ge 0 }}}\frac{t_{0}^{\iii_1}\cdots t_{0}^{\iii_n}}{n!(-z)^{k+1}}
\langle \tau_{0}(\phi_{\iii_1}),\dots,\tau_{0}(\phi_{\iii_n}),\tau_k(\phi_{e(\pmb \iii)_{n+1}})
\rangle^{\un}_{0,n+1}\phi^{e(\pmb \iii)_{n+1}},
\end{multline*}
where the following notation has been used.
\begin{notn}  Any multiindex
with $n$ integer entries
$\pmb \iii\in [0, 4]^n$ can be completed
$$e({\pmb \iii})=\left(\iii_1,\dots,\iii_n,
5\left\{ \frac{-2-\abs{\pmb \iii}}{5}\right\}\right);$$
so that $e(\pmb \iii)$ is a  multiindex
satisfying $\abs{ e(\pmb \iii)}\equiv -2 \mod 5$ (the same selection rule
of \eqref{eq:selrule}).
\end{notn}

In the expression above,
two multiindices $\pmb \iii$
that coincide after a permutation yield the same contribution.
It makes sense to sum with multiplicities
over all nonnegative multiindices $(k_0,\dots,k_4)$
choosing a representant for all permutations of the multiindex
$$\iii({\pmb k})=(\underbrace{0,\dots, 0}_{k_0}, \underbrace{1,\dots, 1}_{k_1},
\dots, \underbrace{4, \dots,4}_{k_4}).$$
The multiplicities are provided by the multinomial coefficients $(k_0+\dots+k_4)!/(k_0!\dots k_4!)$.
Now we can rewrite $J_\un(t,z)$ as
\begin{align*}
J_\un(t,z)&=\sum_{\substack{k_0,\dots, k_4\ge 0}
}J^{\pmb k}_{\un} (t,z),\\
J^{\pmb k}_{\un}(t,z)&=
\sum_{l\ge 0}
\frac{(t^0_0)^{k_0}
\dots (t^4_0)^{k_4}}{z^{l+1}k_0!\dots k_4!}
\langle \underbrace{\tau_{0}(\phi_{0}),\dots,\tau_{0}(\phi_{0})}_{k_0},
\dots,
\underbrace{\tau_{0}(\phi_{4}),\dots,\tau_{0}(\phi_{4})}_{k_4},
\tau_k(\phi_{e({\pmb k})_{n+1}})
\rangle^{\un}_{0,n+1}\phi^{e({\pmb k})_{n+1}},
\end{align*}
where, by a slight abuse of notation, we wrote $e({\pmb k})$
rather than $e(\iii({\pmb k}))$.
Using \eqref{eq:untwistinv} and the fact that $\phi_h=\phi^{3-h}$ for all $h$
(we are in the untwisted state space $H_{\un}$) we get
\begin{equation}\label{eq:Jun}
J_{\un}(t,-z)=z\sum_{k_0,\dots, k_4\ge 0}
\frac{1}{z^{\abs{\pmb k}}}
\frac{(t^0_0)^{k_0}
\dots (t^4_0)^{k_4}}{k_0!\dots k_4!}
{\phi}_{{\sum ik_i}}.
\end{equation}
Indeed, we have $\int \psi^l=\delta_{l,\dim}/5$,
where $\dim$ is the dimension of the moduli stack $\wt A_4$, \emph{i.e.}
$\abs{\pmb{k}}-2$.

Step 2. We derive from $J_{\un}$ a class
of functions taking values in $\Lcal_{\un}$.
\begin{notn}[the functions $\pmb s(x)$ and $G_y(x,z)$]
In $\CC[y,x,z,z^{-1}][[s_0,s_1,\dots]]$, we define
\begin{eqnarray}
G_y(x,z)=\sum_{m,l\ge 0} s_{l+m-1}\frac{B_m(y)}{m!}\frac{x^l}{l!}z^{m-1},
\end{eqnarray}
with $s_{-1}=0$.
By the definition of
Bernoulli polynomials
one gets
\begin{eqnarray}
G_y(x,z)&=&G_0(x+yz,z),\label{eq:Gy}\\
G_0(x+z,z)&=&G_0(x,z)+\pmb s(x)\label{eq:Gs},
\end{eqnarray}
where $\pmb s(x)$ is given by
$$\pmb s(x)=\sum_{d\ge 0} s_d \frac{x^d}{d!}.$$
\end{notn}
\begin{notn}[the dilation vector fields $D_i$]
For $i=0,\dots,4$ we write
\begin{equation}\label{eq:dilation}
D_i=t^i_0 \frac{\partial}{\partial t^i_0},\end{equation}
for the vector field on the state space $H_{\un }=\{
t^0_0\phi_0+t^1_0\phi_{1}+t^2_0\phi_{{2}} +t^3_0\phi_{{3}}
+t^4_0\phi_{{4}}\}$
attached to
$t^i_0$.

The vector field $D_i$ naturally operates on $J_{\un}$.
It is easy to check (using \eqref{eq:Jun})
that
\begin{equation*}
D_i J^{\pmb k}_{\un}=k_iJ^{\pmb k}_{\un}.\end{equation*}
In the following lemma we operate on $J_\un$ by means of the
vector field $$\nabla=\sum_{i=0}^4 \frac{i}{5}D_{i}.$$
\end{notn}
\begin{lem}\label{lem:alpha}
The family $t\mapsto \exp(-G_{\frac{1}{5}}(z\nabla ,z))J_\un(t,-z)$
lies on $\Lcal_{\un}$.
\end{lem}\label{lem}
\begin{proof}
The proof of
\cite[p.11]{CCIT} extends word for word.
The only difference is that
here
$G_{\frac{1}{5}}(z\nabla,z)$
replaces $G_0(z\nabla,z)$; therefore, the coefficient
$B_m{\textstyle \left(\frac{1}{5}\right)}$ should replace $B_m(0)$ in the
definition of the operator
$P_i(z\nabla,z)=
\sum_{m=0}^{i+1} \frac{1}{m!(i+1-m)!} z^m B_m{\textstyle \left(0\right)}
(z\nabla)^{i+1-m}$ used there.

For clarity, we go through the generalization step by step.
Note that $\exp(-G_{\frac{1}{5}}(z\nabla ,z))J^\un(\tau,-z)$ depends on
the variables $\{s_0,s_1,\dots\}$; therefore we write
$$J_{\pmb s}(\tau,-z)=\exp(-G_{\frac{1}{5}}(z\nabla ,z))J^\un(\tau,-z).$$
Remark that we can write the point $J_{\pmb s}(\tau, -z)$ as
$$h=-z+\sum_{k\ge 0} t_k z^k + \sum _{j\ge 0}
\frac{p_j}{(-z)^{j+1}},$$
where $t_k$ is just a parameter varying in the state space $H_{\un}$,
which we can write in terms of parameters $t^k_i$ in the ground ring $R$ as
$t_k=\sum_{k\ge 0}{t^k_i}\phi_i$.
The claim of the lemma is $h\in \Lcal_\un.$ More explicitly
we want to show that the variables $p_k$ are expressed in terms of the
previous variables as follows
\begin{equation}
p_j=\sum_{n\ge 0} \ \sum_{\substack{
0 \le i_1,\dots, i_n\le 4 \\ {a_1, \dots, a_n\ge 0 }}}\ \sum _{0\le \epsilon \le 4} \
\frac{t_{a_1}^{\iii_1}\dots t^{\iii_n}_{a_n}}{n!} \left\langle\tau_{a_1}(\phi_{\iii_1}),\dots,
\tau_{a_n}(\phi_{\iii_n}),\tau_{j}(\phi_{e(\pmb \iii)_{n+1}}) \right\rangle^{\un}_{0,n+1}
\phi^{e(\pmb h)_{n+1}}.\end{equation}
The idea of \cite{CCIT} is to define
$$E_j(h)=p_j-\sum_{n\ge 0} \ \sum_{\substack{
0 \le i_1,\dots, i_n\le 4 \\ {a_1, \dots, a_n\ge 0 }}}\ \sum _{0\le \epsilon \le 4} \
\frac{t_{a_1}^{\iii_1}\dots t^{\iii_n}_{a_n}}{n!} \left\langle\tau_{a_1}(\phi_{\iii_1}),\dots,
\tau_{a_n}(\phi_{\iii_n}),\tau_{j}(\phi_{e(\pmb \iii)_{n+1}}) \right\rangle^{\un}_{0,n+1}
\phi^{e(\pmb h)_{n+1}}.$$
and to prove that for any $j$ the composition
$E_j\circ J_{\pmb s}$ vanishes.

The claim is evident when we set $s_0=s_1=\dots =0$, so we can proceed by induction:
we define the degree of the variable $s_k$ to be $k+1$ and we proceed by induction on the
degree.
Namely we show that
$$\text{$E_j(J_{\pmb s}(\tau, -z))$ vanishes up to degree $n$}\ \ \Longrightarrow\ \
\text{$\frac{\partial}{\partial s_i} E_j(J_{\pmb s}(\tau, -z))$ vanishes up to degree $n$.}$$
The condition on the left hand side above simply means
that we can adjust  $J_{\pmb s}(\tau,-z)$
by adding terms as follows
\begin{equation}\label{eq:Jtilde}
\wt J_{\pmb s}(\tau,-z)=J_{\pmb s}(\tau, -z)+ Y\end{equation}
so that the degree of $Y$ in the variables ${s_1,s_2,\dots}$ is greater than $n$
and $E_j(\wt J_{\pmb s}(\tau,-z))$ vanishes.
(More precisely let us project $J_{\pmb s}(\tau,-z)$ to $[J_{\pmb s}(\tau,-z)]_+\in \Vcal^+_{\un}$ and
decompose $J_{\pmb s}(\tau,-z)$ as $[J_{\pmb s}(\tau,-z)]_-+[J_{\pmb s}(\tau,-z)]_+$.
We remark that in order to match the  desired vanishing condition
we only need to modify the term
$[J_{\pmb s}(\tau,-z)]_-$.)

Let us focus on $\frac{\partial}{\partial s_i} E_j(J_{\pmb s}(\tau, -z))$. By the chain rule, we have
$$\frac{\partial}{\partial s_i} E_j(J_{\pmb s}(\tau, -z))=
d_{J_{\pmb s}(\tau,-z) }E_j \left(\frac{\partial}{\partial s_i} J_{\pmb s}(\tau,-z)\right).$$
By the definition of $G_{\frac15} (x,z))$, we have
\begin{equation}\label{eq:chain}
\frac{\partial}{\partial s_i} E_j(J_{\pmb s}(\tau, -z))=
d_{J_{\pmb s}(\tau,-z) }E_j \left(z^{-1} P_i J_{\pmb s}(\tau, -z)\right),
\end{equation}
where
$$P_i=\sum_{m=0}^{i+1} \frac{1}{m!(i+1-m)!} z^m B_m{\textstyle \left(\frac{1}{5}\right)}
(z\nabla)^{i+1-m}.$$
Note that the right hand side of \eqref{eq:chain}
coincides with
\begin{equation}\label{eq:diff}
d_{\wt J_{\pmb s}(\tau,-z) }E_j \left(z^{-1} P_i \wt J_{\pmb s}(\tau, -z)\right),
\end{equation}
up to degree $n$ in the variables $s_k$.

We need to show that
$\frac{\partial}{\partial s_i} E_j(\wt J_{\pmb s}(\tau, -z))$ vanishes. Recall
that the point $\wt  J_{\pmb s}(\tau, -z)$ lies on $\Lcal_\un$.
We want to show that $P_i \wt J_{\pmb s}(\tau, -z)$ is a point in $zT$,
where $T$ is
the tangent space of $\Lcal_\un$ at  $\wt J_{\pmb s}(\tau, -z)$.
Recall that
$T$ and $\Lcal_{\un}$ satisfy Givental's property
$$zT=\Lcal_\un \cap T.$$
Notice that $P_i$ is a linear combination $\la_0 z^{i+1}+\la_1 z^{i}(z\nabla)+
\dots+\la_{i+1} (z\nabla)^{i+1}$; so we can prove the claim term by term.
First, consider the (possibly iterated) multiplication by $z$.
By the string equation, we can regard  $\wt J_{\pmb s}(\tau, -z)\in\Lcal_\un$ as
 a point of $T$; hence, the multiplication by $z$
sends $T$ to itself ($zT\subset T$) and we have $z^{i+1}(T)\subset zT$.
Second, by the definition of $\nabla$, by operating by $z\nabla$
on  $\wt J_{\pmb s}(\tau, -z)$ we obtain a point of $zT$.
In fact, by $zT=\Lcal_\un \cap T$,
this point lies on $T$ (see \cite[Corollary B.7]{CCIT})
and on $\Lcal_{\un}$; therefore by applying again $z\nabla$
we still land in $zT$.
This allows us to
iterate this argument and finally show that $P_i
\wt J_{\pmb s}(\tau, -z)$ belongs to $zT$; therefore the term between brackets in
the right hand side of \eqref{eq:chain} belongs to $T$. We conclude that \eqref{eq:diff}, and so
$\frac{\partial}{\partial s_i} E_j(\wt J_{\pmb s}(\tau, -z))$ vanishes; hence
we have proven that
$\frac{\partial}{\partial s_i} E_j(J_{\pmb s}(\tau, -z))$ vanishes up to degree $n$.
\end{proof}

Step 3.
We apply to the function
$t\mapsto \exp(-G_{\frac{1}{5}}(z\nabla ,z))J_\un(t,-z)$ the
symplectic transformation
$\Delta\colon \Vcal_{\un}\to \Vcal_{\tw}$
from Proposition \ref{pro:CZ}.
Write $J_{\un}(t,-z)$ as the sum over the multiindices $\pmb k$ of the terms
$J_{\un}^{\pmb k}(t,-z)$. Then we can write
the family of Lemma \ref{lem:alpha}
as
$$t\mapsto \exp(-G_{\frac{1}{5}}(z\nabla ,z))J_\un(t,-z)=
\sum_{\pmb k} \exp\left(-G_{\frac{1}{5}}\left(\frac{\sum_{i=0}^4 ik_{i}}{5}z,z\right)\right)J^{\pmb k}_{\un}(t,-z).$$
By \eqref{eq:Gy}, the above expression equals
\begin{equation}\label{eq:beforeDelta}
\sum_{\pmb k} \exp\left(-G_{0}\left(\frac{1+\sum_{i=0}^4 ik_{i}}{5}z,z\right)\right)J^{\pmb k}_{\un}(t,-z).
\end{equation}By applying
the operator from Proposition \ref{pro:CZ}
$$\Delta=\bigoplus_{i=0}^{4}
\exp \left(\sum_{d\ge 0} s_d \frac{B_{d+1}\left(\frac{i+1}{5}\right)}
{(d+1)!}z^d
\right)\overset{\eqref{eq:Gy}}{=}\bigoplus_{i=0}^{4}
\exp \left( G_0\left(\frac{i+1}{5}z,z\right)
\right)$$ to \eqref{eq:beforeDelta} we get
$$\sum_{\pmb k} \exp\left(G_0\left(\left(\frac{1}{5}+\left \{\frac{\sum  ik_{i}}{5}\right\}\right)
z,z\right)
-G_{0}\left(\frac{1+\sum ik_{i}}{5}z,z\right)\right)J^{\pmb k}_{\un}(t,-z).$$
The relation $\pmb s(x)=G_0(x+z,z)-G_0(x,z)$ yields
$$\sum_{\pmb k}
\exp\left(-\sum_{0\le  m<\lfloor\sum ik_i/5\rfloor}\pmb s\left(\frac{1}{5}z+
\left\{ \frac{\textsum ik_i}{5} \right\} z+ mz\right)\right)J^{\pmb k}_{\un}(t,-z).
$$
In this way, Step 2 and Proposition \ref{pro:CZ} imply the following lemma.
\begin{lem}\label{thm:Itw}
Let $M_{\pmb k}(z)$ be the \emph{modification function}
$$M_{\pmb k}(z)=\exp\left(-\sum_{0\le  m<\lfloor\sum ik_i/5\rfloor}\pmb s\left(-\frac{1}{5}z-
\left\{ \frac{\textsum ik_i}{5} \right\} z- mz\right)\right).$$
Define $I_\tw (t,z)$ as the sum
$$I_\tw (t,z)=\sum_{\pmb k} M_{\pmb k}(z)J_\un ^{\pmb k}(t,z).$$
Then, the family $t\mapsto I_\tw(t,-z)$ lies on $\Lcal_\tw$.\qed
\end{lem}

Step 4. As we pointed out in Remark \ref{rem:fixs}
 the Lagrangian cone $\Lcal_{\RW}$ can
be realized from $\Lcal_\tw$.
This can be done by setting the parameters $s_0,s_1,\dots$ as in
\eqref{eq:fixs}
and by taking the nonequivariant limit.

First, observe that $M_{\pmb k}(z)$ assumes a simple form under the conditions
\eqref{eq:fixs}.
We have \begin{align*}
M_{\pmb k}(z)&=\exp\left(-\sum_{0\le  m<\lfloor\sum ik_i/5\rfloor}\pmb s\left(-\frac{1}{5}z-
\left\{ \frac{\textsum ik_i}{5} \right\} z- mz\right)\right)
\\&=\prod_{0\le  m<\lfloor\sum ik_i/5\rfloor} \exp\left(-s_0-
\sum_{d>0} s_d \frac{\left(-\frac{1}{5}z-
\big\{ \frac{\sum ik_i}{5} \big\} z- mz\right)^d}{d!}\right)\\
&=\prod_{0\le  m<\lfloor\sum ik_i/5\rfloor}\exp\left(5\ln(\lambda)+
\sum_{d> 0} 5(-1)^{d-1} \frac{1}{d}\left(\frac{\frac{1}{5}z+
\big\{ \frac{\sum ik_i}{5} \big\} z+ mz }{\la}\right)^d\right)
\\
&=\prod_{0\le  m<\lfloor\sum ik_i/5\rfloor}\left(\la \exp\left(\ln\left (1+\frac{\frac{1}{5}z+
\big\{ \frac{\sum ik_i}{5} \big\} z+ mz }{\la }\right)\right)\right)^5
\\
&=\prod_{0\le  m<\lfloor\sum ik_i/5\rfloor}\left(\la+\frac{1}{5}z+
\Big\{ \frac{\sum ik_i}{5} \Big\} z+ mz \right)^5\\
&=\prod_{\substack{0\le b<\sum ik_i/5\\ \{ b\}=\{ \sum ik_i/5  \}}}\left(\la+\frac{1}{5}z+ b z \right)^5.
\end{align*}
By using the formula for $J_{\un} ^{\pmb k} (t,z)$
we deduce that
$$\sum_{k_0,\dots, k_{4}\ge 0} \
\prod_{\substack{0\le b<\sum ik_i/5\\
\{ b\}=\{ \sum ik_i/5  \}}}
\left(\la -\frac{1}{5}z- b z \right)^5
\frac{1}{(-z)^{\abs{\pmb k}-1}}
\frac{(t^0_0)^{k_0}
\dots (t^4_0)^{k_4}}{k_0!\dots k_4!}
{\phi}_{\sum ik_i}$$
lies on the cone $\Lcal_{\tw}$ when the parameters $s_d$ are
fixed as above.
In particular, if $t$ ranges over $\{t^i_0=0\mid i\ne 1\}=\{(0,t,0,0,0)\}\subset H_{\tw}$,
we have
$$I_{\tw}(t,-z)=
\sum_{k\ge 0} \
\prod_{\substack{0\le b<k/5\\ \{ b\}=\langle k/5  \rangle}}
\left(\la -\frac{1}{5}z- b z \right)^5
\frac{1}{(-z)^{k-1}}
\frac{t^{k}}{k!}
{\phi}_{k}.$$
Recall that $\phi_0=\phi^3$, $\phi_1=\phi^2$, $\phi_2=\phi^1$, $\phi_3=\phi^0$,
and $\phi_4=\exp(-s_0)\phi^4=\la^5\phi^4$; therefore, in order to
compute the twisted invariants we write $I_{\tw}$ in the form
$$I_{\tw}(t,-z)= \omega_1\phi^3+
\omega_2\phi^2+\omega_3\phi^1+\omega_4\phi^0+\la^5 \omega_5\phi^4,$$
and the fifth term vanishes for $\la \to 0$.
We have
$$\prod_{\substack{0\le b<k/5\\ \{ b\}=\langle k/5  \rangle}}
\left(-\frac{1}{5}z- b z \right)^5= {\displaystyle \frac{\Gamma \left(\frac{k+1}{5}\right)}
{\Gamma(\{ \frac{k+1}{5}\})}(-z)^{5\lfloor k/5\rfloor }}.
$$
In this way
\begin{equation}\label{eq:beforet}
-z\sum_{0\le k \not \equiv 4}
\left(\frac{\Gamma \left(\frac{k+1}{5}\right)}
{\Gamma(\{ \frac{k+1}{5}\})}\right)^5
\frac{1}{(-z)^{k-5\lfloor k/5\rfloor}}
\frac{t^{k}}{k!}
{\phi}^{3-k}\end{equation}
lies on $\Lcal_{\RW}$ for any value of $t$.
Note that, since $\Lcal_{\RW}$ is a cone,
if we modify the above expression by multiplying
by $t$, then we still get a family of elements of $\Lcal_{\RW}$.
In this way, shifting by one the variable $k$ and expressing
the $\Gamma$-functions in terms
of the Pochhammer symbols by $[a]_n=\Gamma(a+n)/\Gamma(a)$,
we get the $I$-function of the statement. \end{proof}

\subsection{Relating the $I$-functions of FJRW theory and GW theory}

\begin{rem}[Proof of Theorem \ref{thm:PF}]
\label{rem:deducing_main}
The functions $\omega_k^{\RW}$ in the statement
of the theorem form a
well known basis of the solution space
of the Picard--Fuchs equation in the variable $t$:
$[{D_t}^4-5^5t^{-5}\prod_{m=1}^4(D_t-mz)]f=0$.
As mentioned in the introduction
the entire calculation of $I_{\RW}$ is parallel to
Givental's calculation \cite{Gi}
of $I_{\GW}$ for the quintic three-fold $X_{W}$.
There, we get
\begin{equation}\label{eq:IforGW}
I_{\GW}(q,z)=
 \sum_{d\ge 0}zq^{H/z+d}
\frac{\prod_{k=1}^{5d}(5H+kz) }{\prod_{k=1}^d (H+kz)^5},
\end{equation}
which (by $H^4=0$)
is a solution of the same equation in the variable $q$:
$[{D_q}^4-5q\prod_{m=1}^4(5D_q+mz)]f=0$.
By expanding $I_{\GW}$ in the variable $H$,
we obtain the
four period integrals spanning the solution space
\begin{equation}\label{eq:IforGWexpanded}
I_{\GW}=\omega_1^{\GW}z+\omega_2^{\GW} {H}+\omega_3^{\GW} \frac{1}{z}H^2+
\omega_4^{\GW}\frac{1}{z^2}H^3.
\end{equation}
This implies Theorem \ref{thm:PF} stated in the introduction.
\end{rem}

\begin{rem}\label{rem:HKQ}
The $I$-functions $I_{\GW}$ and
$I_{\RW}$ may be regarded as multivalued
functions\footnote{The multivaluedness comes from the factor $q^{H/z}$, which should be
read as $\exp(H\log(q)/z)$.}
in the variables
$q$ and $t^5$ (from this point of view, the parameter $H$ should be
considered as a complex variable).
In the literature,
the function $I_{\GW}$ is regarded as a solution of
the Picard--Fuchs equation in the sense that
it satisfies the equation up to
order $4$ in the variable $H$ and the first four terms of the Taylor
expansion yield a basis of the solution space of the
differential equation. The picture at a neighbourhood of $t=0$ is simpler: the elements
of the basis
$\omega_k^{\RW},\ k=1,\dots,4$ admit
the explicit expression given in Theorem \ref{thm:FJRW-I-funct}
and assemble into a regular function in the variable $t$.
In this respect, Theorem \ref{thm:PF} is the
interpretation in terms of enumerative geometry of curves
of a well known analytic picture.

String theory provides another
geometric interpretation of the same analytic picture.
As mentioned in the introduction
there is a physical interpretation of  the
Landau--Ginzburg/Calabi--Yau correspondence in terms of period integrals.
In the framework of mirror symmetry,
this is usually referred to as the B-model picture and
it reflects information on the A-model which is only
partly understood and
incorporates, in particular, the $\GW$ theory of
the quintic three-fold.
From this B-model interpretation, in \cite{HKQ}, Huang, Klemm,
and Quackenbush built upon
physical grounds new predictions in higher genera. Near $q=0$, the
B-model higher genus potential is determined up to $g=51$ and
is expected to match the (A-model)
higher genus $\GW$ invariants of the quintic three-fold.
Near $t=0$, the B-model higher genus potential is explicitly
determined in low genus (see \cite[\S3.4]{HKQ});
$\RW$ theory supplies the A-model geometrical interpretation of these invariants
usually referred to as the ``Gepner point potential''
in the physical literature (in \cite{HKQ} the Gepner point
is called ``orbifold point'' $F_{\rm orb}$ due to the
fact that the point $t=0$ is actually an
orbifold point with stabilizer of order $5$).
The functions $\omega_{\RW}$ of Theorem \ref{thm:FJRW-I-funct} match
\cite[3.57]{HKQ}
up to a constant factor
(due to the
fact that the period integrals have been
rescaled in \cite[(3.64-66)]{HKQ}\footnote{We are grateful to
Klemm for explaining this to us.}).
As we now illustrate,
the genus-zero formulae of \cite[\S3.4]{HKQ} are
matched by our $I_{\RW}$ formula. It would be an extremely interesting problem to compute the higher
genus $\RW$ theory to match \cite{HKQ} Gepner point
potential.
\end{rem}

\begin{rem}[intersection numbers]\label{rem:numbers}
We derive intersection numbers in Givental's formalism.
The function
$I_{\RW}(t,z)$ of the statement of the theorem above can be modified
by multiplication by any function taking values in the ground ring.
We already used this property in the last step of the proof to
get $I_{\RW}$ from \eqref{eq:beforet}.
We can expand \eqref{eq:beforet} in the variable $z$ and
get the four function $\omega_k^{\RW}$:
\begin{equation}\label{eq:FJRWperiods}
\omega_1^{\RW}(t)\phi^3z
+\omega_2^{\RW}(t)\phi^2
+\omega_{3}^{\RW}(t)\phi^1 z^{-1}
+\omega_{4}^{\RW}(t)\phi^0z^{-2}
\end{equation}
with $\omega_1^{\RW}$ invertible; therefore, we can multiply
\eqref{eq:beforet} by $1/\omega_1^{\RW}$ and get a function
lying on the cone $\Lcal_{\RW}$.
In fact, by $\phi_k=\phi^{3-k}$, we have\begin{equation}
J_{\RW}\left( \frac{\omega_{2}^{\RW}}{\omega_1^{\RW}}(t) \phi_1,-z\right)=
-z\phi_0+\frac{\omega_2^{\RW}}{\omega_1^{\RW}}(t) \phi_1
- \frac{\omega_{3}^{\RW}}{\omega_1^{\RW}}(t) \phi_2 z^{-1}
+ \frac{\omega_{4}^{\RW}}{\omega_1^{\RW}}(t) \phi_3 z^{-2};
\end{equation}
indeed
the $J$-function $J_{\RW}(\sum_\iii\tau_\iii\phi_\iii,-z)$ is characterized by being on the cone
$\Lcal_{\RW}$ and
admitting an expression of the form $-\phi_0+\sum_\iii\tau_\iii\phi_\iii+O(-1)$.
If we write $\tau$ for
$(\omega_2^{\RW}/\omega_1^{\RW})(t)$ we can regard the above expression
as the value of $J_{\RW}(\tau\phi_1,-z)$.

We can invert the relation $\tau=(\omega_2^{\RW}/\omega_1^{\RW})(t)$
explicitly in low degree
$$t=\tau -\frac{13}{1125000}\tau ^6-\frac{31991}{97453125000000}
\tau^{11}-\frac{294146129}{9976763671875000000000}\tau^{16}+O(\tau^{21}),$$
plug it into $(\omega_2^{\RW}/\omega_1^{\RW})(t)$, and get the coefficient of $z^{-1}$
in $J_{\RW}(\tau\phi_1,z)$ in low degree in the variable $\tau$:
$${\frac {1}{2}}{\tau}^{2}+{\frac {1}{393750}}{\tau}^{7}+{\frac {239}{
1559250000000}}{\tau}^{12}+{\frac {6904357}{452279953125000000000}}{\tau}^{
17}+O \left( {\tau}^{22} \right).$$
By the definition \eqref{eq:Jlocus} of $J_{\RW}$,
the above power series coincides with the power series
$$\sum_{h\ge 0} \frac{\langle \overbrace{\tau_{0}(\phi_1)\dots \tau_{0}(\phi_1)}^{5h+3}
\rangle^{\RW}_{0,5h+3}}{5h+2!}\tau^{5h+2};$$
therefore,  we get
$$
\begin{tabular}{c||c|c|c|c|c|}
   $n$&$ 3$ &$ 8$ &$ 13$ & $18 $ 
   \\
  \hline&&&&\\
  $\langle {\tau_{0}(\phi_1)\dots \tau_{0}(\phi_1)}
\rangle^{\RW}_{0,n}$ &
{\Large $\frac15$} & {\Large{$\frac {8}{3125}$}} & {\Large${\frac {5736}
{390625}}$} & {\Large${\frac {1325636544}{1220703125}}$}
\end{tabular}\ .
$$
After the identification $s=t/5$ and multiplication by $5$ (due to
the rescaling of the period integrals in \cite{HKQ}),
the generating function
$$\frac15 \frac{t^3}{3!}+ \frac {8}{3125}\frac{t^8}{8!} +
\frac {5736}{390625} \frac{t^{13}}{13!} +
{\frac {1325636544}{1220703125}}\frac{t^{18}}{18!}+\dots$$
matches the genus-zero potential $F_{\rm orb}^0$ near $t=0$
 computed at page 21 of \cite{HKQ}.

A similar calculation yields the coefficient of
$z^{-2}$
in $J_{\RW}(\tau\phi_1,z)$ in low degree:
$$
{\frac {1}{6}}{\tau}^{3}+{\frac {1}{525000}}{\tau}^{8}+{\frac {239}{
1842750000000}}{\tau}^{13}+{\frac {6904357}{508814947265625000000}}{\tau}^{
18}+O \left( {\tau}^{23} \right)
$$
which coincides by \eqref{eq:Jlocus} with the generating function of
the invariants
$\langle {\tau_{0}(\phi_1)\dots \tau_{0}(\phi_1)}
\tau_1(\phi_0)\rangle^{\RW}_{0,5h+4}$.
In this way we get
$$
\begin{tabular}{c||c|c|c|c|c|}
   $n$&$ 4$ &$ 9$ &$ 14$ & $19 $  \\
  \hline&&&&\\
  $\langle {\tau_{0}(\phi_1)\dots \tau_{0}(\phi_1)\tau_{1}(\phi_0)}
\rangle^{\RW}_{0,n}$
&
{\Large$\frac15$}
&{\Large ${\frac {48}{3125}}$}
&{\Large${\frac {63096}{390625}}$}
&{\Large${\frac {21210184704}{1220703125}}$}
\end{tabular}\ \ ,
$$
but these values can be easily deduced from the
previous table via the dilaton equation\footnote{The value corresponding to $n$ in
the first table coincides with
that corresponding to $n+1$ in
the second table after multiplication by $2g-2+n=n-2$ (dilaton equation).}.

As expected (see Remark \ref{rem:CYsimpl})
we only find nonvanishing numbers $\langle \quad \rangle ^{\RW}$ with
insertions of type
$\tau_{0}(\phi_1)$ and $\tau_{1}(\phi_0)$ (the reader may
also refer to
equation \eqref{eq:localindicesforWG} which describes how
for $5$-pointed curves with $4$ points with local index $J^2$
we necessarily have a Ramond markings on the fifth marking---hence, a vanishing
invariant).

Carel Faber's computer programme for computation with
tautological classes together with the formula
of Proposition \ref{pro:GRR} also allowed us to compute explicitly
the first coefficients of the above lists.
Clearly, this algorithm is definitely less efficient than the
one illustrated above. Indeed, in this respect,
Givental's quantization may be regarded
as a setup embodying the GRR algorithm used by Faber's computer programme.
\end{rem}

\begin{cor}\label{cor:sympl}
There is a $\CC[z,z^{-1}]$-valued degree-preserving symplectic transformation $\UU$
 mapping $I_{\RW}$ to the analytic continuation of $I_{\GW}$
    near $t=0$. This proves the genus-zero LG/CY correspondence
    (Conjecture \ref{conj},(1)).
    \end{cor}

\begin{proof} The genus-zero LG/CY correspondence states that the \emph{entire} Lagrangian cones
match after analytic continuation. Notice that both
Lagrangian cones $\Lcal_{\RW}$ and $\Lcal_{\GW}$
are uniquely determined by $I_{\RW}$ and $I_{\GW}$. This happens because
they both are determined by means of tautological relations
starting from the $J$-functions \eqref{eq:Jlocus} which are uniquely deduced from
$I_{\RW}$ and $I_{\GW}$
(this reconstruction property
holds only for the quintic polynomial
and is a consequence of \eqref{eq:dimensioncompatibility}).

We only need to analytically continue the function $I_{\GW}$ near $t=0$ and
derive the change of basis matrix.
By the formula $z^{-l}\prod_{k=1}^l (x+kz) = \Gamma(1+\frac xz +l)/ \Gamma(1+\frac xz)$,
we rewrite $I$ as
\begin{equation}\label{eq:GW_Ifunct}
I_{\GW}(q,z)=
zq^{H/z} \sum_{d\ge 0}q^{d}
\frac{\Gamma(1+5\frac{H}{z}+5d)}{\Gamma(1+5\frac{H}{z})}\frac{\Gamma(1+\frac{H}{z})^5}
{\Gamma(1+\frac{H}{z}+d)^5};
\end{equation}
hence, we have
\begin{align*}\label{eq:sumres}
I_{\GW}(q,z)&=
zq^{H/z} \frac{\Gamma(1+\frac{H}{z})^5}{\Gamma(1+5\frac{H}{z})}\sum_{d\ge 0}q^{d}
\frac{\Gamma(1+5\frac{H}{z}+5d)}
{\Gamma(1+\frac{H}{z}+d)^5},\\
&=zq^{H/z}
\frac{\Gamma(1+\frac{H}{z})^5}{\Gamma(1+5\frac{H}{z})}\sum_{d\ge
0} \frac1{2\pi\cxi}\Res_{s=d}\frac{1}{e^{2\pi \cxi s}-1}
\frac{\Gamma(1+5\frac{H}{z}+5s)} {\Gamma(1+\frac{H}{z}+s)^5}q^{s}.
\end{align*}
Consider the contour integral
\begin{equation}
zq^{H/z}
\frac{\Gamma(1+\frac{H}{z})^5}{\Gamma(1+5\frac{H}{z})}\int_C
\frac{1}{e^{2\pi \cxi s}-1} \frac{\Gamma(1+5\frac{H}{z}+5s)}
{\Gamma(1+\frac{H}{z}+s)^5}q^{s},\label{eq:contourint}
\end{equation}
where the $C$ is curve shown below.
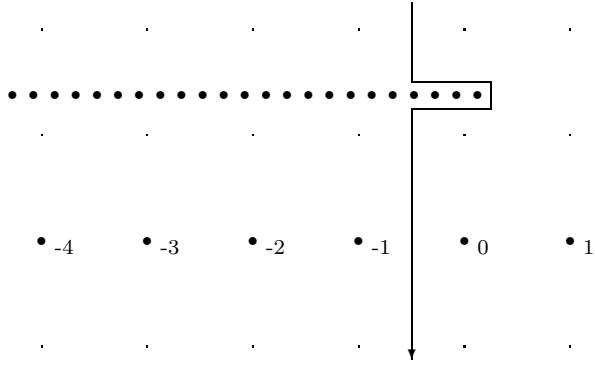
\begin{figure}[h]\label{fig:contour}
\begin{picture}(200,150)(-200,-10)
  \qbezier[5](-100,0)(0,0)(100,0)
  \qbezier[5](-100,40)(0,40)(100,40)
  \qbezier[5](-100,80)(0,80)(100,80)
  \qbezier[5](-100,120)(0,120)(100,120)
  \put(60,40){\circle*{3}}
  \put(65,35){\footnotesize{0}}
  \put(100,40){\circle*{3}}
  \put(105,35){\footnotesize{1}}
  \put(20,40){\circle*{3}}
  \put(25,35){\footnotesize{-1}}
  \put(-20,40){\circle*{3}}
  \put(-15,35){\footnotesize{-2}}
  \put(-60,40){\circle*{3}}
  \put(-55,35){\footnotesize{-3}}
  \put(-100,40){\circle*{3}}
  \put(-95,35){\footnotesize{-4}}
  \put(65,95){\circle*{3}}
  \put(57,95){\circle*{3}}
  \put(49,95){\circle*{3}}
  \put(41,95){\circle*{3}}
  \put(33,95){\circle*{3}}
  \put(25,95){\circle*{3}}
  \put(17,95){\circle*{3}}
  \put(9,95){\circle*{3}}
  \put(1,95){\circle*{3}}
  \put(-7,95){\circle*{3}}
  \put(-15,95){\circle*{3}}
  \put(-23,95){\circle*{3}}
  \put(-31,95){\circle*{3}}
  \put(-39,95){\circle*{3}}
  \put(-47,95){\circle*{3}}
  \put(-55,95){\circle*{3}}
  \put(-63,95){\circle*{3}}
  \put(-71,95){\circle*{3}}
  \put(-79,95){\circle*{3}}
  \put(-87,95){\circle*{3}}
  \put(-95,95){\circle*{3}}
  \put(-103,95){\circle*{3}}
  \put(-111,95){\circle*{3}}
%
  \put(40,90){\vector(0,-1){95}}
  \put(40,90){\line(1,0){30}}
  \put(70,90){\line(0,1){10}}
  \put(70,100){\line(-1,0){30}}
  \put(40,100){\line(0,1){30}}
\end{picture}
\caption{the contour of integration.}
\end{figure}
By \cite[Lem.~3.3]{Ho}, the integral \eqref{eq:contourint}
is defined and analytic throughout the region $\abs{\arg(-q)}<\pi$.
For $\abs{q}<1/5^5$, we can close the contour to the right; then
\eqref{eq:contourint} equals \eqref{eq:GW_Ifunct}.
For $\abs{q}>1/5^5$, we can close the contour to the left; then
\eqref{eq:contourint} is the sum of residues at
$$s=-1-l,  \ l\ge 0\qquad\qquad {\rm and} \qquad \qquad s=-\frac{H}{z}-\frac{m}{5}, \ m\ge  1.$$
The residues at $s=-1-l$ vanish as they are multiples of $H^4=0$.
Thus the analytic continuation of $I_{\GW}(q,z)$ is given by
evaluating the remaining residues.
Since we have
$$\Res_{ s=-\frac{H}{z}-\frac{m}{5}}\Gamma\left(1+5\frac{H}{z}+5s\right)=-\frac{1}5\frac{(-1)^m}{\Gamma(m)},$$
we get
$$I'_{\GW}(q,z)=
\frac{z}5\frac{\Gamma^5(1+\frac{H}z)}{\Gamma(1+5\frac{H}z)}
\sum_{{0<m\not \equiv 0}}(-1)^{m}\frac{(2\pi\cxi)\xi^m}{e^{-2\pi \cxi \frac{H}z}-\xi^m}
\frac{q^{-\frac{m}{5}}}
{\Gamma(m)\Gamma(1-\frac{m}{5})^5}, \qquad \text{where $\xi=e^{\frac{2\pi\cxi}{5}}$.}
$$
Note that in the above equation we have omitted the terms corresponding to $m\in 5\ZZ$,
because in these cases the
residues vanish.
In terms of the coordinate $t$, which satisfies $t^5=q^{-1}$, we have
(using $\Gamma(x)\Gamma(1-x)=\pi/\sin(\pi x)$)
\begin{align*}
I'_{\GW}(t,z)=&\frac{z}5\frac{\Gamma^5(1+\frac{H}z)}{\Gamma(1+5\frac{H}z)}
\sum_{{0<m\not \equiv 0}}(-1)^{m}
\frac{(2\pi\cxi)\xi^m}{e^{-2\pi \cxi \frac{H}z}-\xi^m}\frac{t^{{m}}}
{\Gamma(m)\Gamma^5(1-\frac{m}{5})}\\
=&\frac{z}5\frac{\Gamma^5(1+\frac{H}z)}{\Gamma(1+5\frac{H}z)}
\sum_{{0<m\not \equiv 0}}
(-1)^{m}\frac{(2\pi\cxi)\xi^m}{e^{-2\pi \cxi \frac{H}z}-\xi^m}
t^m\frac{\Gamma^5(\frac{m}{5})}
{\Gamma(m)}\frac{\pi^5}{\sin^5(\frac{m}{5}\pi)}\\
=&\frac{z}5\frac{\Gamma^5(1+\frac{H}z)}{\Gamma(1+5\frac{H}z)}
\sum_{k=1,2,3,4}
(-1)^k
\frac{(2\pi\cxi)\xi^k}{e^{-2\pi \cxi \frac{H}z}-\xi^k}
\frac{1}{\Gamma^5(\frac{k}{5})\Gamma^5(1-\frac{k}{5})}
\sum_{l\ge 0}
t^{k+5l}\frac{\Gamma^5(\frac{k+5l}{5})}
{\Gamma(k+5l)}\\
=&\frac{z}5\frac{\Gamma^5(1+\frac{H}z)}{\Gamma(1+5\frac{H}z)}
\sum_{k=1,2,3,4}
\frac{(-1)^k}{\Gamma^5(1-\frac{k}{5})}
\frac{(2\pi\cxi)\xi^k}{e^{-2\pi \cxi \frac{H}z}-\xi^k}
\sum_{l\ge 0}
\left(\frac{\Gamma(\frac{k+5l}{5})}{\Gamma(\frac{k}{5})}\right)^5
\frac{t^{k+5l}}{\Gamma(k+5l)}\\
=&\frac{z}5\frac{\Gamma^5(1+\frac{H}z)}{\Gamma(1+5\frac{H}z)}
\sum_{k=1,2,3,4}
\frac{(-1)^k}{\Gamma^5(1-\frac{k}{5})}
\frac{(2\pi\cxi)\xi^k}{e^{-2\pi \cxi \frac{H}z}-\xi^k}
\omega_k^{\RW}.
\end{align*}
The Taylor expansions of the functions
\begin{equation*}\label{eq:gf}
g(\rho)=\frac{\xi^k}{e^{-2\pi \cxi \rho}-\xi^k}\qquad \text{ and }\qquad
f(\rho)=\frac{\Gamma^5(1+\rho)}{\Gamma(1+5\rho)}\end{equation*}
read
\begin{equation*}\label{expg}
g(\rho)=\frac{\xi^k}{1-\xi^k} +\frac{\xi^k}{(1-\xi^k)^2}(2\pi\cxi)\rho+
\frac{\xi^k(1+\xi^k)}{2(1-\xi^k)^3}(2\pi\cxi)^2\rho^2+
\frac{\xi^k(1+4\xi^k+\xi^{2k})}{6(1-\xi^k)^4}(2\pi\cxi)^3\rho^3+O(\rho^4).\end{equation*}
and $$f(\rho)=1+C(2\pi\cxi)^2\rho^2-E(2\pi\cxi)^3\rho^3+O(\rho^4),$$
where $C=5/12$ and $E=-\zeta(3)40/(2\pi\cxi)^3$ (with $\zeta(3)$ equal to Ap\'ery's constant)
are related
to the intersection theory
of the quintic three-fold ($C=c_2(X_W)[H]/\deg(W)24$ and $E=\zeta(3)\chi(X_W)/\deg(W)(2\pi\cxi)^3$).

In this way, by expanding $I'_{\GW}$ in the variable ${H}$, we
get $\wt \omega_1^{\GW}, \wt \omega_2^{\GW}, \wt \omega_3^{\GW}, \wt \omega_4^{\GW}$, the analytic continuations of
$\omega_1^{\GW}, \omega_2^{\GW}, \omega_3^{\GW}, \omega_4^{\GW}$ from \eqref{eq:IforGW}. After
identifying $\phi^k$ with $\fie^k=[H]/5$, the change of basis matrix from
$\omega_1^{\RW}, \omega_2^{\RW}, \omega_3^{\RW}, \omega_4^{\RW}$
to $\wt \omega_1^{\GW}, \wt \omega_2^{\GW}, \wt \omega_3^{\GW}, \wt \omega_4^{\GW}$
is defined column by column as
\begin{equation}\label{eq:FJRWtoGW}
\mathbb U=\begin{pmatrix}\
\frac{(-1)^k(2\pi\cxi)}{\Gamma^5(\frac{5-k}{5})} \frac{\xi^k}{1-\xi^ k} z^{k-1}& \\
\frac{(-1)^k(2\pi\cxi)^2}{\Gamma^5(\frac{5-k}{5})} \frac{\xi^k }{(1-\xi^k)^2} z^{k-2}& k=1,2,3,4\\
\frac{(-1)^k(2\pi\cxi)^3}{\Gamma^5(\frac{5-k}{5})} \left(\frac{\xi^k(1+\xi^k)}{2(1-\xi^k)^3}+C\frac{\xi^k}{(1-\xi^k)}\right) z^{k-3}&
\\
\frac{(-1)^k(2\pi\cxi)^4}{\Gamma^5(\frac{5-k}{5})} \left(\frac{\xi^k(1+4\xi^k+\xi^{2k})}{6(1-\xi^k)^4}+C\frac{\xi^k}{(1-\xi^k)^2}-E\frac{\xi^k}{1-\xi^k}\right)z^{k-4}&
\end{pmatrix},
\qquad \text{where $\xi=e^{\frac{2\pi\cxi}{5}}$.}\end{equation}
Clearly $\mathbb U$ is degree-preserving
with $\deg z=2$.
It can be easily seen via a direct computation
that $\mathbb U$ is symplectic
(\emph{i.e.} we have $\mathbb U^*(-z)\mathbb U(z)=\pmb 1$)\footnote{Choosing
a suitable basis on GW side, \cite{HKQ} shows that
this transformation can be defined entirely inside ${\rm Sp}(4,\QQ(\xi))$. See also
Iritani \cite{Ir} where a formalism based on Grothendieck--Riemann--Roch for
orbifolds is proposed (this uses the above interpretation of
$C$ and $E$ in terms of intersection theory).}.
\end{proof}

\appendix
\setcounter{thm}{0}
\section{Orbifold curves and roots}\label{appendix}
We recall that an
orbifold curve
is a $1$-dimensional stack
of Deligne--Mumford type with nodes as singularities,
a finite number of ordered markings, and possibly
nontrivial
stabilizers only at the markings and at the nodes.
The action on a local parameter $z$ at the markings is given by
$z\mapsto \xi_l z$ (for $l\in \ZZ_{\ge 1})$,
whereas the action on a node $\{xy=0\}$ is
$(x,y)\mapsto (\xi_k x,\xi_k^{-1} y)$ (for $k\in \ZZ_{\ge1}).$
Such a curve $C$ is naturally equipped with a (locally free)
sheaf of logarithmic differentials
$\omega_{\log}$;
\emph{i.e.} the sheaf of sections of the relative dualizing sheaf $\omega$
possibly with poles of order $1$ at the markings. It is natural to
refer to $\omega_{\log}$ rather than $\omega$, because it corresponds to the
pullback of the sheaf of logarithmic differentials from the coarse curve.
Above, $k$ and $l$ denote two orders of stabilizers. It
is automatic for a Deligne--Mumford stack-theoretic curve to have
cyclic stabilizers; however, in the above expression for the local
action at a node, we require that
the product of the factors multiplying $x$ and $y$ is $1$ (this
insures that orbifold curves can be smoothed).
An orbifold curve with $n$ markings is called an $n$-pointed orbifold curve;
the above definition naturally extends to the notion of
``family of orbifold curves over the base scheme $X$'' (or simply
``orbifold curve over $X$''): these are flat morphisms $C\to X$
of relative dimension one from a Deligne--Mumford stack $C$ to $X$ for which we extend
the local descriptions given above (we refer the reader to \cite[Defn.~4.1.2]{AV}).

Note that two stabilizers arising in
an orbifold curve may have different orders; this
feature prevents the moduli stack of orbifold curves from being separated
(see \emph{e.g.} \cite{chstab}) and
motivates the following definition imposing a stability condition
yielding a proper (and separated) stack.
\begin{defn}\label{defn:dstable}
For any
positive integer $d$, a \emph{$d$-stable} $n$-pointed genus-$g$
curve is a proper and geometrically connected orbifold curve $C$ of
genus $g$ with $n$ distinct smooth markings $\sigma_1, \dots, \sigma_n$ such
that
\begin{enumerate}
\item
the corresponding coarse $n$-pointed curve is stable;
\item
all stabilizers (at the nodes and at the markings) have order $d$.
\end{enumerate}
The moduli stack $\MMM_{g,n,d}$ classifying $n$-pointed genus-$g$
$d$-stable curves is proper and smooth and
has dimension $3g-3+n$.
(It differs from the moduli stack of Deligne--Mumford stable
curves only because of the stabilizers over the normal
crossings boundary divisor, see the discussion in
\cite[Thm.~4.1.6]{chstab}.)
\end{defn}

The moduli functor of $d$th roots
via the above notion of $d$-stable curve
is straightforward.
For any nonnegative integer $d$,
the category $\RRR_d$ of $d$th roots of $\omega_{\log}$ over
genus-$g$ $n$-pointed $d$-stable curves, is fibred
over $\MMM_{g,n,d}$
\begin{equation}\label{eq:droots}
\RRR_d\to \MMM_{g,n,d}
\end{equation}
and it is formed by
the following objects: a $d$-stable curve,
$(C\to X, \sigma_1\dots,\sigma_n\in C)$
equipped with a line bundle $L\to C$ and
an isomorphism of line bundles over $C$
$$\fie\colon L^{\otimes d}\to \omega_{\log}.$$
A morphism from $(C'\to X';\sigma'_1,\dots,\sigma'_n\in C'; L'; \fie')$ to
$(C''\to X'';\sigma''_1,\dots,\sigma''_n\in C'';  L'', \fie'')$
is a morphism $\al$ between the
$d$-stable curves and an isomorphism of line bundles
$\rho\colon L'\to \al^*L''$ compatible with $\fie'$ and $\fie''$; \emph{i.e.}
we have $\al^*\fie''\circ \rho^{\otimes d}=\fie'$.
The above category is a proper moduli stack, and the above map
$\RRR_d\to \MMM_{g,n,d}$
is an \'etale cover of $\MMM_{g,n,d}$.
If $n=0$, we have $\RRR_d\neq \varnothing$ if and only if
$d$ divides $2g-2$. For $n>0$, the stack is always nonempty and
the fibre over $\MMM_{g,n,d}$ is the union of $d^{2g-1+n}$ copies of $B\pmmu_d$;
\emph{i.e.} the above morphism has degree $d^{2g-2+n}.$

\begin{rem}
On a $d$-stable curve $C$,
we often need to twist line bundles
by divisors supported on the markings.
For a marking $\sigma$ which has stabilizer $G$ of order $d$ it is convenient to
denote by $({1}/{d})D$ the divisor
$BG\hookrightarrow C$ of degree $1/{d}$ and to write $D$ for
the pullback of the underlying point
on the coarse curve. We refer to $D$ as the
\emph{integer divisor corresponding to a
marking}. This is a good spot to recall the following well known fact, which
uses this convention.
\end{rem}

\begin{lem}\label{lem:exists}
Let $C$ be a $d$-stable
curve (over $\CC$) and let $M$ be a line bundle pulled back from the coarse space
(\emph{e.g.} $\omega_{\log}$ and its tensor powers).
There is an equivalence between two categories
of $d$th roots $L$ on
$d$-stable curves:
\begin{multline*}
\{L\mid L^{\otimes d}\cong M\}
\longleftrightarrow\\
\bigsqcup _{0\le E< d\sum_{i=1}^n{D_i}} \{L\mid L^{\otimes d}\cong M(-E)
\text{ and the stabilizers at the markings
act trivially on $L$}\},\end{multline*} where the
union on the right hand side ranges over all divisors $E$
which are linear combinations of integer divisors $D_i$ corresponding to the markings
with multiplicities in  $\{0,\dots, d-1\}$.

In particular as soon as $n$ is positive or $d$ divides $2g-2$
we have $\RRR_d\neq \varnothing$. Furthermore, in these cases $\RRR_d
\to\MMM_{g,n,d}$ is proper and \'etale.
The fibres of these morphisms to $\MMM_{g,n,d}$
are all
isomorphic to a zero-dimensional stack given by
$d^{2g-1+n}$ copies of $B\pmmu_{d}$ if $n>0$ and $d^{2g}$ copies of $B\pmmu_{d}$
if $n=0$ and $2g-2\in d\ZZ$.
\end{lem}
\begin{proof}
The correspondence is given by the functor $L\mapsto p^*p_*L$ where $p$
is the map forgetting the stabilizers along the markings.
The claim follows easily from \cite{chstab}:
the morphism $\RRR_d
\to\MMM_{g,n,d}$ is proper and all fibres are isomorphic.
A direct analysis of the fibre over a
point representing a smooth curve yields the result.
\end{proof}

{\small

\vspace{.5cm}

\noindent \textsc{Institut Fourier, UMR du CNRS 5582,
Universit\'e de Grenoble 1,
BP 74, 38402,
Saint Martin d'H\`eres,
France}\\
\textit{E-mail address:} \url{chiodo@ujf-grenoble.fr}

\vspace{.3cm}

\noindent \textsc{Department of Mathematics, University of Michigan, Ann Arbor, MI 48109-1109,
USA} and  \textsc{Yangtze Center of Mathematics, Sichuan University, Chengdu,
610064, P.R. China}\\
\textit{E-mail address:} \url{ruan@umich.edu}
}
\end{document}